\documentclass[12pt,reqno]{amsart}

\usepackage{color, amsmath,amssymb, amsfonts, amstext,amsthm, latexsym}

\allowdisplaybreaks
\newcommand{\lb}{\langle}
\newcommand{\rb}{\rangle}
\def\divv{{\rm div}\,}

\newcommand{\rA}{\mathrm{A}}

\newcommand{\rH}{\mathrm{ H}}

\newcommand{\rV}{\mathrm{ V}}
\newcommand\dela[1]{}

\newcommand{\eps}{\varepsilon}

\numberwithin{equation}{section}
\newtheorem{theorem}{Theorem}[section]

\newtheorem{lemma}[theorem]{Lemma}
\newtheorem{remark}[theorem]{Remark}
\newtheorem{proposition}[theorem]{Proposition}
\newtheorem{corollary}[theorem]{Corollary}

\newtheorem{Prop}[theorem]{Proposition}

\newtheorem{assumption}[theorem]{Assumption}

\begin{document}

\title[Splitting for 2D NSE]{
Splitting up method for the  2D stochastic Navier-Stokes equations }
\author[H. Bessaih]{H. Bessaih}
\address{University of Wyoming, Department of Mathematics, Dept. 3036, 1000
East University Avenue, Laramie WY 82071, United States}
\email{ bessaih@uwyo.edu}
\author[Z. Brze\'zniak]{Z. Brze\'zniak}
\address{University of York, Department of Mathematics,
Heslington, York YO10 5DD, UK}
\email{zdzislaw.brzezniak@york.ac.uk}
\author[A. Millet]{A. Millet}
\address{SAMM, EA 4543,
Universit\'e Paris 1 Panth\'eon-Sorbonne, 90 Rue de
Tolbiac, 75634 Paris Cedex France \textit{ and} Laboratoire de
Probabilit\'es et Mod\`eles Al\'eatoires,
  Universit\'es Paris~6-Paris~7, Bo\^{\i}te Courrier 188,
      4 place Jussieu, 75252 Paris Cedex 05, France}
\email{annie.millet@univ-paris1.fr \textit{ and}
annie.millet@upmc.fr}
\date{\today}

\begin{abstract}
 In this paper, we deal with the convergence of an iterative scheme for
the 2-D stochastic Navier-Stokes Equations on the torus
 suggested by the Lie-Trotter product formulas
for stochastic differential equations of parabolic type.
The stochastic system  is split into two problems which are simpler for numerical computations.
An estimate of the approximation error is given  for
periodic boundary conditions.
In particular,
 we prove
  that the strong speed of the convergence  in probability
is almost $1/2$. This is shown by means of an $L^2(\Omega,\mathbb{P})$ convergence localized on a set of arbitrary large
probability. The assumptions on the diffusion coefficient
depend on the fact that some multiple of the Laplace
operator is present or not with the multiplicative stochastic term. Note that if one of the splitting steps only
contains the stochastic integral, then
the diffusion coefficient
may not contain any gradient of the solution.
\end{abstract}

\maketitle
 \maketitle

\maketitle
\section{Introduction}\label{sec-intro}

Assume that  $D$ is a regular bounded open domain of $\mathbb{R}^{2}$ with boundary $\partial{D}$ and let $\mathbf{n}$  denote
the external normal field to the boundary.  Let us  consider the Navier-Stokes Equations with, for concreteness, the Dirichlet boundary conditions.

\begin{equation}\label{NStokes}
\frac{\partial u(t,x)}{\partial t}-\nu\Delta u(t,x)+(u(t,x)\cdot\nabla) u(t,x)+\nabla p(t,x)=G(t,u(t,x))\dot{W}(t,x),
\end{equation}
$t\in [0,T]$, $x\in D$, with the incompressibility condition
\begin{equation}\label{incomp}
\nabla\cdot u(t,x)=0,\quad t\in [0,T], \quad x\in D,
\end{equation}
the boundary  condition
\begin{equation}\label{bound}
 u=0 \mbox{ on } \partial D
\end{equation}
and the initial condition
\begin{equation}\label{init}
u(0,x)=u_{0}(x) \quad x\in D.
\end{equation}

Here, $u$ is the velocity, $p$ is the pressure, $\nu$ is the viscosity coefficient, $W$ is a cylindrical Brownian motion and $G$ is an operator valued function  acting on divergence free vector fields. Details  will be given in the next sections.

Our  main results from sections \ref{sect-periodic} and \ref{sec-convergence}
are  valid  only  for the stochastic Navier-Stokes equations with the periodic boundary conditions. Our results from section \ref{sec3}
are valid for  for the stochastic Navier-Stokes equations with both the periodic and Dirichlet boundary conditions, 
see section \ref{sect-periodic} for details.  \dela{Moreover, in the latter case, our convergence rate are better than in the general case, see \ref{sec-convergence}.}

 The well-posedness of the system \eqref{NStokes}-\eqref{incomp} has been extensively investigated.
Under very general conditions on the operator $G$ we refer to \cite{Flandoli_Gatarek_95} for martingale stationary solutions;
the uniqueness of the $2$-dimensional case has been investigated in \cite{Schmalfuss_1997}.
For the strong solutions and more general models in hydrodynamic we refer to \cite{ChuMill}.
 For a comprehensive setting and higher dimension, we refer to \cite{Fla} and the references therein.

In this paper, we deal with the convergence of some iterative schemes suggested by the Lie-Trotter product formulas for
 the stochastic differential equations of parabolic type.
 The stochastic system  is split into two problems which are simpler for numerical computations.
  Other numerical schemes have been used in the literature.
  This method has been used for purely theoretical purposes of proving the existence
   of solutions of stochastic geometric heat equation  in \cite{Funaki_1992}.
 Let us mention two recent papers on the topic of numerical approximations to 2D stochastic Navier-Stokes Equations.

 Carelli and Prohl \cite{CarPro}   studied the strong speed of convergence of some
 fully and implicit semi-discrete Euler time  and finite element based space-time discretization  schemes.
  As in \cite{Prin} for the stochastic Burgers equation, they  proved the
 $L^2(\Omega, {\mathbb P})$ convergence of the error term localized by a set (depending on the discretization),
 whose probability converges to   $1$ as the
 time mesh decreases to $0$. They assume  that the initial condition $u_0 \in L^8(\Omega,V)$, and that the diffusion coefficient of
 the multiplicative noise  may contain some gradient of the solution.  These authors studied  the speed of convergence in probability
 (introduced by Printemps in \cite{Prin}) of the fully implicit scheme to the solution in two different spaces
 $C([0,T];H)\cap L^2(0,T;V)$  and $C([0,T];V')\cap L^2(0,T;H)$ (here $V$ and $H$ are the usual functional spaces used
 for the Navier-Stokes equations, they will be defined later on).
 They showed that these are equal respectively  to $1/4 - \alpha$ and  $1/2-\alpha$
 for some positive $\alpha$.  Finally, they showed that the difference between the fully and the semi-implicit schemes
converges to $0$ in probability in $C([0,T];V')\cap L^2(0,T;H)$ with the speed $1/2-\alpha$.
The speed of convergence of some space-time
discretization was also investigated. Note that in \cite{CarPro}, no projection on divergence-free fields is made and the pressure term
is part of the discretized process.

Using the semi-group and cubature techniques, D\"orsek \cite{Dorsek} studied the weak speed of    convergence of a certain
Strang  time-splitting scheme combined  with a Galerkin approximation in the
space variable for the SNSEs with an additive noise. Two equations were solved  alternatively on each time interval of the subdivision
 to approximate the Galerkin approximation:
one deterministic equation only contains the bilinear term and the other one is the stochastic heat equation.
The   weak speed of  convergence of this  approximation was given in terms of the largest eigenvalue $N$ of the space
of eigenvectors of the Stokes operator where the projection is done: if  a function $\varphi$ belongs to ${\mathcal C}^6(L^2(D))$
with some exponential control of its derivatives of order  up to $6$,
then the speed of convergence of the error term between the Strang time splitting $v^N_n$  with time mesh $T/n$ of
the Galerkin approximation $u^N$, that is
$\big| \mathbb{E}[\varphi(u^N(t))] - \mathbb{E}[\varphi(u^N_n(t))]\big|$
is   estimated from above in terms of $N$ and $n$.
Finally, let us note that in that paper the noise is both additive and finite-dimensional and that the initial condition again belongs  to the space  $V$.

 One of the aims of the current  paper is to show that a certain  splitting up method for the stochastic Navier-Stokes Equations
 with multiplicative noise is convergent. Our  splitting method is
 implemented using on each time interval two consecutive steps. In  the first step,  the deterministic Navier-Stokes Equations
 (with modified viscosity) is solved. The corresponding solution is denoted by  $u^n$. In the second step the linear parabolic
Stokes equation (again with  a modified viscosity) with a stochastic perturbation is solved. The corresponding solution is denoted by  $y^n$.
The goal of the paper  is twofold. On one hand, we  establish the "strong" speed of convergence in $L^2(D)$  uniformly
 on the time  grid after the two steps have been performed, that is for the difference $|y^n(t_k^-)-u(t_k)|$.
  We furthermore prove that the same speed of convergence
holds in $ L^2(0,T; V)$
 for $u^n-u$ and $y^n-u$.
 The regularity assumptions on the diffusion coefficient and
 on the initial condition are similar to that in \cite{CarPro}. For instance we  assume  $\mathbb{E}\Vert u_0\Vert _{V}^8<\infty$,
 but either the topology we use is sharper for a similar speed of convergence or  the speed of convergence is doubled in the same topological space.
While the first step consists in solving the deterministic Navier-Stokes
equations,  the next step consists in only computing the stochastic integral (with no smoothing Stokes operator), no gradient
is allowed in the diffusion coefficient and in that case the regularity of the corresponding process $y^n$ is weaker than that
of the terminal process.
Following the definition  introduced in \cite{Prin}, we
deduce that the speed of convergence in probability (in the above functional spaces) of this splitting scheme is almost $1/2$.
When the second step of the splitting contains some part of the
viscosity together with the stochastic integral,
due to the smoothing effect of the Stokes operator, the diffusion coefficient may contain some gradient terms
provided that the stochastic parabolicity condition holds. In fact the convergence of the scheme requires a stronger
control of the gradient part in the diffusion coefficient. The main result states an $L^2(\Omega, \mathbb{P})$ convergence
of the error term localized on a set which depends on the solution and whose probability can be made as close to 1 as required.
Furthermore, for technical reasons, the speed of convergence is obtained for $z^n-u$ in $L^\infty(0,T;H) \cap L^2(0,T;V)$,
where $z^n$ is a theoretical process which is defined in terms of $u^n$ and $y^n$ and such that $z^n(t_k)=y^n(t_k^-)=u^n(t_{k+1}^+)$.

In  section 7  of a paper \cite{BCP} by the first named authour with Carelli and Prohl,
 the convergence of the finite element approximation for the $2$-D stochastic Navier-Stokes Equations
  by using the local monotonicity trick of Barbu
was proved. The use of the Barbu trick allowed the authours to identify the strong solution without using
  a compactness method. In this respect that paper is similar to ours. However,
  the time-splitting approximation we use in the current paper as well as the proofs are completely different.

The strong convergence of the splitting method has already been studied in a series of papers by Gy\"ongy and Krylov (\cite{GK},
\cite{GK2})  but for parabolic stochastic PDEs with some degenerate stochastic parabolicity condition.
However, the linear setting used in these papers does not cover the hydrodynamical models used in the present one.
The splitting method was  also studied in \cite{BrzMil} for SPDEs including the
stochastic linear Schr\"odinger equations.

Section \ref{sec2} describes the properties on the hydrodynamical models, the noise, the assumptions on the diffusion coefficients
ensuring well-posedeness of the solution in various functional spaces. It also describes
 both splitting schemes, depending on the fact that some multiple of the Stokes operator is kept
with the stochastic integral or not. Section \ref{sec3} proves several a priori bounds such as the control of the norms of the approximating
processes independently of the time mesh, as well as the control  of the difference of both processes introduced in the steps of the
algorithm in terms of the time mesh.
Note that the abstract properties of the operators  needed  to prove the results in  sections \ref{sec2} and \ref{sec3}
are satisfied not only by the stochastic
2D Navier-Stokes Equations, but also by various hydrodynamical models, such as the Boussinesq, magneto-hydrodynamical  and
B\'enard models; the framework also contains the shell and dyadic models of turbulence (see e.g. \cite{ChuMill}).
Sections  \ref{sect-periodic} and \ref{sec-convergence} focuse on the 2D Navier-Stokes Equations on the torus ;
they require some further properties involving the Stokes operator and the non linearity, more regularity on the initial condition
and additional properties on the "diffusion" coefficient $G$.
Section \ref{sect-periodic} provides further a priori bounds "shifting" the spatial regularity.  Section \ref{sec-convergence}  proves the
$L^2(\Omega, {\mathbb P}) $ convergence of a localized version of both algorithms with some explicit control on the
constant used in the localization. This enables us to deduce the speed of convergence in probability.

As usual, we denote by $C$ (resp. $C(T)$)  a positive constant (resp. a positive constant depending on $T$)
which may change from line to line, but does not depend on $n$.
\bigskip

\section{Preliminaries and assumptions}\label{sec2}

\subsection{Functional setting}\label{AB}

We at first describe the functional setting and the operators of the 2D Navier-Stokes Equations. As in \cite{Temam_2001} we denote by   $\mathcal{V}$
the space of infinitely differentiable divergence free  and compactly supported vector
fields $u$ on $ D$.

Let the Hilbert space $(H, |\cdot|) $  be the closure of  ${\mathcal V}$
in the $\mathbb{L}^{2}=L^2(D,\mathbb{R}^2)$ space.
Let also  $ E$ be the Hilbert space consisting  of all $ u\in \mathbb{L}^2$ such that $\divv u \in\,L^2(D)$.
It is known, see  \cite[Theorem I.1.2]{Temam_2001}, that
there exists a unique bounded linear map $\gamma_{\mathbf{n}}: E\to H^{-\frac12}(\partial(D))$,
where $H^{-\frac 12}(\partial{D})$ is  the dual space of $H^{1/2}(\partial{D})$ (and equal to  the image in $L^2(\partial{D})$
 of the trace operator $\gamma_0:H^{1,2}(D) \to L^2(\partial{D})$)  such that
\begin{equation}\label{Temam_1.18}
\gamma_{\mathbf{n}}(u)= \mbox{ the restriction of } u\cdot \mathbf{n} \mbox{ to } \partial{D}, \;\; \mbox{ if } u\in \mathcal{C}^\infty(\overline{D}).
\end{equation}
Then, it is known that $H$ is equal to the space of all vector fields $u \in E$ such that $\gamma_{\mathbf{n}}(u)=0$.
Let us denote by  $V$ the  separable Hilbert space which is equal to the   closure in the space $\mathbb{H}^{1,2}=H^{1,2}(D,\mathbb{R}^2)$ of
the space $\mathcal{V}$ equipped  with the inner product inherited from $\mathbb{H}^{1,2}$ and with corresponding  norm denoted by
 $\| ~\cdot~\|$. Identifying $H$ with its
dual space $H^\prime$, and $H^\prime$ with the corresponding natural subspace of
the dual space $V^\prime$, we have the standard Gelfand triple $V\subset H\subset
V^\prime$ with continuous dense embeddings. Let us note that the duality  pairing
between $V$ and $V^\prime$ agrees with the  inner product of $H$.

Moreover, we set $D(A)=\mathbb{H}^{2,2}\cap V$ and we define the linear operator
$A: D(A)\subset H\longrightarrow H$ as $Au=-P \Delta u$, where $P: \mathbb{L}^2\to H$ is the orthogonal projection called the
Leray-Helmholtz projection. It is known, see \cite{Cattabriga_1961}, that  $A$ is  self-adjoint, positive and has a compact inverse.

The fractional powers of the operator $A$ will be denoted by $A^\alpha$, $\alpha \in \mathbb{R}$.
It is known that  $V$ coincides with $D(A^{1/2})$ with equivalent norms and in what follow  we can use on  $V$  the norm
$\Vert u\Vert = |A^{1/2} u|  =\sqrt{\int_D \vert \nabla u(x) \vert^2\, dx}$.

Let $b(\cdot,\cdot,\cdot): V\times V\times V\longrightarrow
\mathbb{R}$ be the continuous trilinear form defined by
$$b(u,v,z)=\int_{D}(u\cdot\nabla v)\cdot z,$$
which, by the incompressibility condition satisfies
$$b(u,v,v)=0, \;\; u,v\in V.$$
By the Riesz Lemma  there exists a continuous bilinear map
$B: V\times V\longrightarrow V^\prime$ such that
\begin{equation}\label{eqB}
\lb B(u,v),z\rangle =b(u,v,z),\ {\rm for}\ {\rm all}\ u,v,z\in V,
\end{equation}
which also satisfies
\begin{equation}\label{skew}
 \lb B(u,v),z \rb =-\lb B(u,z),v \rb \quad {\rm and}\quad \lb B(u,v),v \rb =0\quad  u, v, z\in V.
\end{equation}
Moreover, as it has been pointed out by V. Barbu \cite{Barbu_2011_private} and proved in \cite[Proposition 2.2]{Sri+Sundar_2006} the following Assumption  is  satisfied with the space $\mathrm{X}=\mathbb{L}^4(D)$.

\begin{assumption}\label{ass-Barbu}
\begin{enumerate}
\item There exists a Banach space $\mathrm{X}$ such that $V\subset\mathrm{X}\subset H$ continuously and densely  and there
exists a constant $C>0$ such that
\begin{equation}\label{inter}
\vert u\vert^{2}_{\mathrm{X}}\leq C|u|\Vert u\Vert, \quad u\in V.
\end{equation}
\item For $\eta>0$, there exists a constant $C_{\eta}>0$ such that for all $u_i \in V$, $i=1,2,3$,
\begin{align} \label{B0}
|\lb B(u_{1}, u_{2}), u_{3}\rb|&\leq  C \vert  u_1\vert_\mathrm{X} \vert  u_2\vert_\mathrm{X} \Vert  u_3\Vert.
\end{align}
 \end{enumerate}
\end{assumption}

\begin{remark}
It is easy to see  the Assumption  \ref{ass-Barbu} (2)   implies that for any $\eta>0$ there exists $C_{\eta}>0$ such that for all $u_i \in V$, $i=1,2,3$,
\begin{align}\label{B1}
|\lb B(u_{1}, u_{2}), u_{3}\rb|&\leq  \eta \Vert  u_{3}\Vert ^{2}
+C_{\eta}\vert u_{1}\vert^{2}_{\mathrm{X}}\vert u_{2}\vert^{2}_{\mathrm{X}},
\end{align}
Moreover, the last property  together with \eqref{skew} implies that for all $(u_i)_{i=1}^3 \in V$,
\begin{equation}\label{Blip}
|\lb B(u_{1}, u_{1})- B(u_{2}, u_{2}), u_{1}-u_{2}\rb|\leq \eta\Vert u_{1}-u_{2}\Vert ^{2}
+C_{\eta}|u_{1}-u_{2}|^{2}\vert u_{1}\vert^{4}_{\mathrm{X}}.
\end{equation}
\end{remark}

Note that this abstract setting (where one only assumes that the operators $A$ and $B$ satisfy the conditions in subsection \eqref{skew} -
\eqref{Blip})
 includes several classical hydrodynamical settings subject to random perturbations, such as the  2D Navier-Stokes Equations,
 the 2D Boussinesq model for B\'enard convection,  the 2D Magneto-Hydrodynamics equations,  the 2D magnetic B\'enard problem,
  the 3D  Leray $\alpha$-model for Navier-Stokes Equations,
 the GOY and Sabra shell models of turbulence, and dyadic models; see \cite{ChuMill} for more details.
  To be precise we assume that
  \begin{trivlist}
  \item[(i)]
  $(H,|.|)$ is a Hilbert space,
  \item[(ii)]
  $A$ is a linear positive unbounded operator in $H$,
  \item[(iii)] $V= D(A^{1/2})$
   endowed with  the norm $\Vert \cdot \Vert $ is a Hilbert space,
   \item[(iv)] a bilinear map $B:V\times V \to V^\prime$ satisfies \eqref{skew}
  \item[(v)]  Assumption \ref{ass-Barbu} holds
  for some Banach space $\mathrm{X}$.
  \end{trivlist}
  In the sequel, we will work in this abstract framework containing the above example of the 2D Navier-Stokes Equations.

We assume that $K$
is a    separable Hilbert space,  $(\Omega, \mathcal{F}, (\mathcal{F})_{t\geq 0}),
\mathbb{P})$ is a filtered probability space and $W=(W(t))_{t\geq 0}$ is a
$K$-cylindrical Wiener process on that probability space.

Hence, the stochastic hydrodynamical systems (including the stochastic 2D Navier-Stokes Equations) are rewritten in the abstract form

 \[\left\{
\begin{array}{ll}
du+(Au+B(u,u)-f)dt=G(u)dW,\\
u(0)=u_{0}.
\end{array}
\right.
\]
For simplicity we will assume that $f=0$.

 We also introduce a Coriolis type of term $R: [0,T]\times H\longrightarrow H$, for example
$R(t, (u_1,u_2))=c_0 (-u_2,u_1)$ (precise assumptions are given below), and  set for $t\in [0,T]$ and $u\in V$:
 $$F(t,u)=Au+B(u,u)+R(t,u)$$
and consider the evolution equation.

\begin{equation}\label{SNS}
d u(t) + F(t,u(t)) dt = G(u(t)) \,dW(t), \quad u(0)=u_0.
\end{equation}

\subsection{Assumptions and  results on the stochastic NSEs in $H$.} \label{subsec-G1R1}
Through the paper, we will assume that  $u_{0}\in H$ (or $u_{0}\in
L^{2}(\Omega,\mathcal{F}_{0}, H)$). Let us denote by ${\mathcal T}_{2}(K,H)$ the space of Hilbert-Schmidt operators from $K$ to $H$.

{\bf Assumption (G1)}: Let us assume that $G$ is a continuous mapping
$G:[0,T]\times V\longmapsto {\mathcal T}_{2}(K,H),$  (resp $G:[0,T]\times H\longmapsto {\mathcal T}_{2}(K,H)$ for $\eps =0$)
 and that there exist positive constants $(K_{i})_{i=0}^3$ and $(L_i)_{i=1}^2$,
 such that for any $t\in [0,T]$, $u,u_1,u_2\in V$,
\begin{eqnarray}\label{assumptionG1}
|G(t,u)|^{2}_{{\mathcal T}_{2}(K,H)}&\leq& K_{0}+K_{1}|u|^{2} + \eps \; K_2 \Vert u\Vert ^2,  \\
\label{assumptionG2}
|G(t,u_{2})-G(t,u_{1})|^{2}_{{\mathcal T}_{2}(K,H)}&\leq & L_{1}|u_{2}-u_{1}|^{2} +  \eps\;  L_2 \Vert u_2 - u_1\Vert ^2.
\end{eqnarray}

{\bf Assumption (R1)}: Let us assume that $R$ is a continuous mapping
$R: [0,T]\times H\longrightarrow H$ such that for some positive constants  $R_{0}$ and $R_{1}$
\begin{equation}\label{R}
|R(t,0)|\leq R_{0},\quad |R(t,u)-R(t,v)|\leq R_{1}|u-v|,\quad u, v\in H.
\end{equation}

There is huge literature for the well posedness of 2D Stochastic Navier-Sokes equations. The following result is similar to that
proved in \cite{CarPro} for the 2D Navier-Stokes Equations.   We refer to
\cite{ChuMill}  for the result given below. Indeed, the a priori estimates in Proposition A.2 of \cite{ChuMill}  imply that the sequence
$u_{n,h}$ of Galerkin approximations satisfies $\sup_n \sup_{h\in {\mathcal A}_M }
\mathbb{E}\big( \sup_{t\in [0,T]} |u_{n,h}(t)|^{2p}\big) <\infty$ and the subsequence used in Step 1 of the proof of Theorem 2.4 in \cite{ChuMill}
can be supposed to be weak-star convergent to $u_h$ in $L^{2p}(\Omega , L^\infty(0,T;H))$.   In our case, there is no random
control, that is $h=0$.

\begin{theorem}\label{thm-nse}
Let us assume that  assumptions {\bf (G1)} and {\bf (R1)} are satisfied with  $ K_2 \leq L_2 <2$.
Then, for any $T>0$ and  any ${\mathcal F}_0$-measurable $H$-valued random variable such that  $\mathbb{E}|u_0|^4<\infty$,
there exists  a unique adapted process  $u$ such that
$$u\in C([0,T]; H)\cap L^2(0,T;V) \cap L^{4}(0,T ; \mathrm{X})\quad {\rm  a.s.}$$
and $P$-a.s., $u$ is solution of the equation \eqref{SNS}, that is in the weak formulation, for all $t\in [0,T]$ and $\phi\in D(A)$:
\begin{equation}\label{weak}
 \begin{split}
 \lb u(t), \phi\rb&+\int_{0}^{t}\lb u(s), A\phi\rb\,ds +
 \int_{0}^{t}\lb B(u(s), u(s)) , \phi  \rb\,ds +
 \int_{0}^{t}\lb R(s,u(s)), \phi\rb\,ds \\
 &=\lb u_{0}, \phi\rb+
 \int_{0}^{t}\lb G(s,u(s))dW, \phi\rb,
 \end{split}
 \end{equation}
Moreover, if  $q\in [2,1+\frac{2}{K_2})$,  then  there exists a positive constant $C=C_q(T)$
 such that if $\mathbb{E}|u_0|^{q}<\infty$,
\begin{equation}\label{thm-estimate}
 \mathbb{E}\Big(\sup_{t\in [0,T] }|u(t)|^{q} +\int_{0}^{T}\! \Vert u(s)\Vert ^{2}(1+|u(s)|^{q-2})\,ds \Big) \leq C(1+\mathbb{E} |u_{0}|^{q}).
\end{equation}

Finally, if  $K_2<\frac23$, then
\begin{equation}\label{thm-estimate4}
 \mathbb{E} \int_{0}^{T}\! \vert u(s)\vert_{\mathrm{X}}^{4}\,ds
  \leq C(1+\mathbb{E} |u_{0}|^{4}).
\end{equation}

\end{theorem}

\subsection{Description of the scheme}\label{desscheme}

Let $\Pi=\left\{ 0=t_{0}<t_{1}<\dots<t_{n}=T\right\}$ be a finite
partition of a given interval $[0,T]$ with constant mesh $h=T/n$. We will consider the following splitting scheme similar to one
introduced in \cite{BGR_92} for deterministic NSEs.
Let $ \eps\; \in [0,1)$ and let $F_\eps : [0,T]\times V \to V^\prime$ be defined by:
\begin{equation} \label{Fepsilon}
F_\eps (t,u) = (1-\eps) Au + B(u,u) + R(t,u).
\end{equation}
Note that $F_0=F$.

{\bf Step 1.}
  Set $t_{-1}=-\frac{T}{n}$.  For $t \in [ t_{-1},0)$ set $y^n(t) = u^n(t) = u_0$  and
 ${\mathcal F}_t={\mathcal F}_0$ .

The scheme $(y^n, u^n) $ is defined by induction as follows. Let $i=0, \cdots n-1$ and  suppose we have defined processes
$u^n(t)$ and $y^n(t)$ for $t\in [t_{i-1}, t_{i})$ such that
$y^n(t_{i}^{-})$ is an $H$-valued $\mathcal{F}_{t_{i}}$-measurable
function. This clearly holds for $i=0$. Then we define $u^n(t),\ \ t\in [t_{i}, t_{i+1})$ as
the unique solution of the (deterministic) problem with positive viscosity $1-\eps$  and
with "initial" condition $ y^n(t_{i}^{-})$ at time $t_i$, that is :
\begin{equation}\label{deterministic-inductive}
\begin{array}{ll}
\frac{d}{dt}u^n(t) + F_{\eps}(t, u^n(t))=0, \; t\in [t_i, t_{i+1}), \quad
u^n(t_{i}^{+})=y^n(t_{i}^{-}).
\end{array}
\end{equation}
Note that  $u^n(t_{i+1}^{-})$ is a well defined
$H$-valued $\mathcal{F}_{t_{i}}$ measurable random variable. Thus we can
define $y^n(t),\ \ t\in [t_{i}, t_{i+1})$ as the unique solution
of the (random) problem with "initial" condition the ${\mathcal F}_{t_i}$-measurable random variable
$u^n(t_{i+1}^-)$ at time $t_i$:
\begin{equation} \label{stochastic-inductive}
dy^n(t)+ \eps\; A y^n(t)dt=G(t,y^n(t))\,dW(t), \;  t\in [t_{i},t_{i+1}),
\quad y^n(t_{i}^{+})=u^n(t_{i+1}^{-}).
\end{equation}
Indeed, if $ \eps\; >0$,  it is classical that, provided that the stochastic parabolicity condition holds
(that is $ \eps\; K_2$, $ \eps\; L_2$ are small enough), there exists a unique  weak solution to \eqref{stochastic-inductive},
see e.g. \cite{KrRo}. If $ \eps\; =0$,
the smoothing effect of $A$ does not act anymore, but the coefficient $G$ satisfies the usual growth and Lipschitz conditions for the $H$-norm.
Therefore, in both cases  $y^n(t_{i+1}^{-})$ is a well defined $H$-valued
$\mathcal{F}_{t_{i+1}}$ measurable random variable.

Finally, let
  $ u^n(T^+)=y^n(T^-)$.
\begin{remark}
As in \cite{BGR_92}, we notice that the processes $u^n$ and  $y^n$  are not continuous.
\end{remark}

In order to prove the convergence of the above scheme, we will need to establish a priori estimates on $u^n$ and $y^n$.
They will be different if $ \eps\;=0$ and $ \eps\; \in (0,1)$. We at first introduce some notation.
Recall that  $\Pi=\left\{ 0=t_{0}<t_{1}<\dots<t_{n}=T\right\}$. Set
\[
\left\{ \begin{array}{lll}
d_{n}(t):=t_{i}, & d^{*}_{n}(t):=t_{i+1}& \mbox{\rm for } t\in [t_{i},t_{i+1}), \; i=0, 1,\dots, n-2, \\
d_{n}(t):=t_{n-1}&d^{*}_{n}(t):=t_{n}& \mbox{\rm for }  t\in [t_{n-1},t_{n}].
\end{array}
\right.
\]

With the above notations, the processes $u^n$ and $y^n$ can be rewritten as follows:
For every $t\in [0,T]$ we have:
\begin{align}\label{equality_u_2}
u^n(t)&=u_{0}- \int_{0}^{t}F_{\eps}(s,u^n(s))ds
+\int_{0}^{d_{n}(t)}\left[ - \eps\; A y^n(s)ds+G(s,y^n(s))dW(s)\right],\\
\label{equality_y_2}
y^n(t)&=u_{0}- \int_{0}^{d_{n}^{*}(t)}F_\eps(s,u^n(s))ds+\int_{0}^{t}
\left[ - \eps\; A y^n(s)ds+G(s,y^n(s))dW(s)\right].
\end{align}
\bigskip

\section{A priori estimates for the initial data in $H$ and general stochastic hydrodynamical systems} \label{sec3}

We at first suppose that the conditions {\bf (G1)} and {\bf (R1)} are satisfied.
\begin{lemma}\label{l1_S2}
Let $u_0$ be ${\mathcal F}_0$-measurable and such that $\mathbb{E}|u_0|^2<\infty$, and fix $ \eps\; \in [0,1)$.
Let Assumptions {\bf (G1)} and {\bf (R1)}
be satisfied with   $ K_2  \leq L_2 <2 $.
Then  there exists a positive constant $C$, depending on $\mathbb{E}|u_0|^2 , \,  \eps\;,\, T$
and the constants $  R_i$ and $K_i $ such that for every
integer $n\geq 1$:
\begin{equation}\label{E1_S2}
 \sup_{ t\in [0, T]} \mathbb{E}\Big(|y^n(t)|^{2} + \sup_{s\in [d_n(t), d^*_n(t)) }|u^n(s)|^{2}  \Big)
+  \mathbb{E} \int_{0}^{T}\Vert u^n(s)\Vert^{2}ds \leq C.
\end{equation}

Furthermore, if   $ \eps\;\in (0,1)$ we can  choose the constant $C$ such  that:
\begin{equation}\label{E3y_S2}
\sup_{n\in\mathbb{N}} \mathbb{E}\int_{0}^{T}\Vert y^n(s)\Vert^{2}ds\leq C.
\end{equation}
\end{lemma}

\begin{proof}
Let $\alpha=2R_{0}+K_{0}$ and  $a=2(R_{0}+R_{1})+K_{1}$.
We use an induction argument to prove that for $l=-1, \cdots, n-1$,
\begin{equation}\label{boundHinduction}
 \mathbb{E} \Big( \sup_{t\in [t_l, t_{l+1})}  |u^n(t)|^2 \Big) +  \sup_{t\in [t_l, t_{l+1})}  \mathbb{E}|y^n(t)|^2 \leq
\mathbb{E} |u_{0}|^{2}{\rm e}^{\frac{(l+1)aT}{n}}+
\frac{\alpha T}{n}\sum_{j=1}^{l+1}{\rm e}^{\frac{j aT}{n}},
\end{equation}
where we use the convention  $\sum_{j=p}^{q}{\rm e}^{\frac{jK_{1}T}{N}}=0$ if $p>q$.
Note that \eqref{boundHinduction} clearly holds for $l=-1$. Assume that \eqref{boundHinduction}
holds for $l=-1, 0, \cdots, i-1$.
Take the scalar product of \eqref{deterministic-inductive} by $u^n$
and integrate over $(t_i,t]$ for $t\in [t_i, t_{i+1})$; the anti-symmetry property of $B$ and
Assumption {\bf (R1)} yield
\begin{align} \label{majouni}
|u^n(t)|^{2} &+ 2(1-\eps)\int_{t_{i}}^{t}
\Vert  u^n(s)\Vert ^2 \,ds =|y^n(t_{i}^{-})|^{2} \nonumber \\
&-2 \int_{t_{i}}^{t}\lb B(u^n(s),u^n(s)),u^n(s)\rb\,ds
+2 \int_{t_{i}}^{t}\lb R(s,u^n(s)),u^n(s)\rb\,ds\nonumber \\
\leq & |y^n(t_{i}^{-})|^{2} + 2 R_0 \frac{T}{n} + 2 \int_{t_i}^t (R_0+R_1) |u^n(s)|^2\,ds.
\end{align}

The It\^o Lemma  yields for $t\in [t_i, t_{i+1})$:
\begin{align}\label{majoyni}
\mathbb{E}|y^n&(t)|^{2}+ 2 \;\eps\; \mathbb{E}\int_{t_{i}}^{t}
\Vert  y^n(s)\Vert ^2 \,ds  =\mathbb{E}|u^n(t_{i+1}^{-})|^{2}
+2\mathbb{E}\int_{t_{i}}^{t}
|G(s,y^n(s))|^{2}_{{\mathcal T}_{2}(K,H)}ds\nonumber \\
&\leq \mathbb{E}|u^n(t_{i+1}^{-})|^{2}+\frac{K_{0}T}{n}+K_{1}\int_{t_{i}}^{t}
\mathbb{E}|y^n(s)|^{2}ds +  \eps\; K_2 \mathbb{E}\int_{t_{i}}^{t}
\Vert  y^n(s)\Vert ^2 \,ds  .
\end{align}
Since $K_2<2$ we may neglect the integrals of the $V$-norm in both inequalities \eqref{majouni} and \eqref{majoyni}.
Taking expected values in \eqref{majouni}, using the Gronwall Lemma and the induction hypothesis, we obtain:
\begin{align*}
  \mathbb{E}\Big( &\sup_{t_{i}\leq t <t_{i+1}} |u^n(t)|^{2}\Big)  \leq \left(\mathbb{E}|y^n(t_{i}^{-})|^{2}+\frac{2R_{0}T}{n}\right)
{\rm e}^{\frac{2(R_{0}+R_{1})T}{n}}\\
&\leq \mathbb{E} |u_{0}|^{2}{\rm e}^{\frac{aiT}{n}+ \frac{2(R_{0}+R_{1})T}{n}}
+\frac{2R_{0}T}{n}{\rm e}^{\frac{2(R_{0}+R_{1})T}{n}}
+\frac{\alpha T}{n}\sum_{j=1}^{i}{\rm e}^{\frac{j aT}{n}+\frac{2(R_{0}+R_{1})T}{n} } ,
\end{align*}
and
\begin{align*}
\sup_{t_{i}\leq t <t_{i+1}} & \mathbb{E}|y^n(t)|^{2}\leq \left(\mathbb{E}|u^n(t_{i+1})|^{2}+\frac{K_{0}T}{n}\right)
{\rm e}^{\frac{K_{1}T}{n}}\\
&\leq \mathbb{E} |u_{0}|^{2}{\rm e}^{\frac{iaT+ aT}{n} }+ \frac{2R_0 T}{n} e^{\frac{aT}{n}} + \frac{K_{0}T}{n}{\rm e}^{\frac{K_{1}T}{n}}
+\frac{\alpha T}{n} e^{\frac{aT}{n}} \sum_{j=1}^{i}{\rm e}^{\frac{j aT}{n}} .
\end{align*}
Since $\alpha\geq 2R_{0}$ and $a\geq 2(R_{0}+R_{1})$,  we deduce that the induction hypothesis \eqref{boundHinduction}
 holds true for $l=i+1$.  Hence we deduce that
\begin{align*}
\sup_{0\leq t< T} \mathbb{E}\Big( \sup_{d_n(t) \leq s < d^*_n(t)} |u^n(s)|^{2}\Big) & \vee
\left(\sup_{0\leq t<T}\mathbb{E}|y^n(t)|^{2}\right)\leq
|u_{0}|^{2}{\rm e}^{\frac{n aT}{n}}+
\frac{\alpha T}{n}\sum_{j=1}^{n}{\rm e}^{\frac{j aT}{n}}\\
&\leq \mathbb{E} |u_{0}|^{2}{\rm e}^{aT}+
\frac{\alpha T}{n}{\rm e}^{\frac{aT}{n}}
\frac{{\rm e}^{a T}-1}{{\rm e}^{\frac{a T}{n}}-1}\,
\leq \,  \mathbb{E}|u_{0}|^{2}{\rm e}^{a T}+
\frac{\alpha}{a}{\rm e}^{2a T}.
\end{align*}
This proves part of \eqref{E1_S2}. Using this last inequality in \eqref{majoyni}, and in \eqref{majouni} after taking expected values,
  yields for every $i=0,\cdots, n-1$:
\begin{align*}
\mathbb{E}| u^n(t_{i+1}^-)|^2 + 2(1-\eps) \mathbb{E}\int_{t_i}^{t_{i+1}} \Vert  u^n(s)\Vert ^2\,ds &\leq \mathbb{E}|y^n(t_{i}^{-})|^{2}
+\frac{CT}{n}, \\
\mathbb{E}|y^n(t_{i+1}^-)|^{2}+ (2-K_2)  \eps\; \mathbb{E}\int_{t_{i}}^{t_{i+1}}\Vert  y^n(s)\Vert ^2 \,ds &\leq \mathbb{E}|u^n(t_{i+1}^{-})|^{2}+\frac{CT}{n}.
\end{align*}
Adding all these inequalities from $i=0$ to $i=n-1$ concludes the proof of \eqref{E1_S2} and proves  \eqref{E3y_S2} when $ \eps\; >0$.
\end{proof}

As usual, we can use higher moments of the $H$-norms and deduce the following.
\begin{lemma}\label{l4_S2}
Fix $ \eps\; \in [0,1)$ and let Assumptions {\bf (G1)} and {\bf (R1)} be satisfied with  $L_2<2$ and $K_2<\frac{2}{2p-1}$
for some real number $p\geq 2$. Then
there exists a positive constant $C:=C( \eps\;,T)$ such that for every integer $n\geq 1$:
\begin{equation}\label{E4u_S2}
\sup_{t\in [0,T]}\Big[ \mathbb{E}\Big(\sup_{s\in [d_n(t),d^*_n(t))} |u^n(s)|^{2p}\Big) +\mathbb{E}|y^n(t)|^{2p}\Big]
+\mathbb{E}\int_{0}^{T}\Vert u^n(s)\Vert ^{2}|u^n(s)|^{2(p-1)}ds \leq C.
\end{equation}
Furthermore, if $ \eps\;\in (0,1)$ then $C$ can be chosen such that:
\begin{equation}\label{E4y_S2}
\sup_{n\geq 1} \mathbb{E}\int_{0}^{T}\Vert y^n(s)\Vert ^{2}|y^n(s)|^{2(p-1)}ds \leq C.
\end{equation}
\end{lemma}
\begin{proof}
 By repeating an argument similar to that used in the proof of Lemma \ref{l1_S2} that we can find in 
 \cite{Temam_2001} (Lemma 1.2), that is:
\[  
\frac{d}{dt}|u^n(t)|^{2}=2\langle \left(u^n(t)\right)', u^n(t)\rangle, 
\]  
and using the chain rule,  we have that for $t\in [t_i, t_{i+1})$, $i=0, \cdots, n-1$: 
\begin{align*}
\frac{d}{dt}|u^n(t)|^{2p}&=
p\left(|u^n(t)|^{2}\right)^{p-1}\frac{d}{dt} |u^n(t)|^{2}\\
&=2p |u^n(t)|^{2(p-1)}\langle \left(u^n(t)\right)', u^n(t)\rangle .
\end{align*}

Hence, we get  that  for $i=0,\dots,n-1$ and for all $t\in [t_{i},t_{i+1})$
\begin{align}\label{u4}
|u^n(t)|^{2p} & + 2p (1-\eps) \int_{t_i}^t |u^n(s)|^{2(p-1)} \Vert u^n(s)\Vert ^2\,ds =| y^n(t_{i}^{-})|^{2p}\nonumber \\
& - 2p\  \int_{t_i}^t\left[ \lb B(u^n(s),u^n(s)),u^n(s)\rb + \lb R(s,u^n(s)),u^n(s)\rb \right]  |u^n(s)|^{2(p-1)}\,ds\nonumber \\
\leq  & |y^n(t_{i}^{-})|^{2p}
+2p\frac{R_{0}T}{n} +2p(R_{0}+R_{1})\int_{t_{i}}^{t}  |u^n(s)|^{2p}ds.
\end{align}
Thus, the Gronwall Lemma implies
\begin{equation*}
\mathbb{E}\sup_{t_{i}\leq t<t_{i+1}} |u^n(t)|^{2p}\leq\left[\mathbb{E}|y^n(t_{i}^{-})|^{2p}+2p\frac{R_{0}T}{n}\right]
{\rm e}^{2p(R_{0}+R_{1})\frac{T}{n}}.
\end{equation*}
On the other hand, for fixed $i=0,\dots,n-1$,  the It\^o Lemma yields for   $t\in [t_{i},t_{i+1})$
\begin{align*}
|&y^n(t)|^{2p} + 2p \;\eps\;\int_{t_{i}}^{t}|y^n(s)|^{2(p-1)}\Vert y^n(s)\Vert ^{2}\,ds =|u^n(t_{i+1}^{-})|^{2p} \\
& +p \int_{t_{i}}^{t}\lb G(s,y^n(s))dW(s),y^n(s)\rb |y^n(s)|^{2(p-1)}\\
&+p \int_{t_{i}}^{t} \! \Big( |G(s,y^n(s))|^{2}_{{\mathcal T}_{2}(K,H)} |y^n(s)|^{2(p-1)}
\! +2(p-1) |G^{*}(s,y^n(s)) y^n(s)   |^{2}_{K} |y^n(s)|^{2(p-2)} \Big)\,ds .
\end{align*}
Using assumption {\bf (G1)}, we deduce:
\begin{align}\label{y4}
 \mathbb{E}|& y^n (t)|^{2p}+2p  \;\eps\;\int_{t_{i}}^{t}\mathbb{E}|y^n(s)|^{2}\Vert y^n(s)\Vert ^{2(p-1) } \,ds \leq
\mathbb{E}|u^n(t_{i+1}^{-})|^{2p}\nonumber \\ &+p(2p-1) \Big[  \frac{K_{0}T}{n}
+(K_0+K_1) \int_{t_{i}}^{t}\mathbb{E}|y^n(s)|^{2p} ds +
 K_2  \;\eps \int_{t_{i}}^{t}\mathbb{E}|y^n(s)|^{2(p-1)} \Vert y^n(s)\Vert ^{2}\,ds\Big].
\end{align}
Since $K_2 < \frac{2}{2p-1}$, in inequality \eqref{y4} we may neglect the integrals containing  the $V$-norms of $u(t)$ and by applying
the Gronwall Lemma we then infer that
\begin{equation*}
\sup_{t_{i}\leq t<t_{i+1}}\mathbb{E}|y^n(t)|^{2p}\leq \left(\mathbb{E}|u^n(t_{i+1}^{-})|^{2p}+\frac{2p(p-1)K_{0}T}{n}\right)
{\rm e}^{\frac{2p(p-1)(K_{0}+K_{1})T}{n}}.
\end{equation*}
Let us put
$$b:=2p (R_{0}+R_{1})+p(2p-1) (K_{0}+K_{1})\; {\rm and}\; \beta:=2p R_{0}+p(2p-1)K_{1}; $$
then by a mathematical  induction argument (see  the proof of Lemma \ref{l1_S2} for a similar one), we infer that  for $i=0,\dots,n-1$,
\[
\mathbb{E} \Big( \sup_{t_{i}\leq t<t_{i+1}} |u^n(t)|^{2p}\Big) \; \vee \; \Big( \sup_{t_{i}\leq t<t_{i+1}}\mathbb{E}|y^n(t)|^{2p}\Big) \leq
\mathbb{E} |u_{0}|^{2p}{\rm e}^{(i+1)\frac{T}{n}b}+\frac{\beta T}{n}\sum_{j=1}^{i+1}{\rm e}^{(jb)\frac{T}{n}}.
\]
Hence we deduce:
\begin{equation}\label{uy4}
\Big( \sup_{0\leq t<  T}  \mathbb{E}\Big[ \sup_{s\in [d_n(t),d^*_n(t))} |u^n(s)|^{2p}\Big] \Big) \;  \vee \;  \Big( \sup_{0\leq t<  T}  \mathbb{E}|y^n(t)|^{2p}\Big)\leq
\mathbb{E} |u_{0}|^{2p}{\rm e}^{bT}+\frac{\beta }{b}{\rm e}^{2bT}.
\end{equation}
Adding \eqref{u4} and \eqref{y4} for $i=1,\dots,n-1$ and using \eqref{uy4}, we deduce that
\begin{align*}
\mathbb{E}|&y^n(T^{-})|^{2p} +p[2-(2p-1)K_2] \,  \eps\;\int_{0}^{T}\!\! \! \mathbb{E}|y^n(s)|^{2(p-1)}\Vert y^n(s)\Vert ^{2}ds
\\
&+  2p(1-\eps)\int_{0}^{T}\!\! \mathbb{E}|u^n(s)|^{2(p-1)}\Vert u^n(s)\Vert ^{2}ds  \leq
\mathbb{E} |u_{0}|^{2p} (1+bTe^{bT}) +\beta T(1+e^{2bT}).
\end{align*}
This completes the proof using once more the fact that $(2p-1)K_2<2$ when $ \eps\; >0$.
\end{proof}

 Finally we will   prove  an upper  estimate
 of the  $H$ norm of the difference of  both processes $u^n$ and $y^n$.
\begin{proposition}\label{prop_y-u}
Let us assume that $u_0$ is ${\mathcal F}_0$-measurable such that ${\mathbb E }|u_0|^4 <\infty$ and that the Assumptions {\bf (G1)} and {\bf (R1)} hold with $K_2<\frac23$  and $L_2<2$.
Then for any  $ \eps\;\in [0,1)$,  there exists a positive constant $C$ such that for any $n\in\mathbb{N}$
\begin{equation}\label{y-u}
\mathbb{E}\int_{0}^{T}|u^{n}(t)-y^{n}(t)|^{2}dt\leq \frac{CT}{n}.
\end{equation}
\end{proposition}

\begin{proof}
{\bf Case 1:}
Let  $ \eps\; =0$; then  \eqref{stochastic-inductive} and Assumption {\bf (G1)} prove that for any $t\in [0,T)$,
\[
\mathbb{E}|y^n(t)-u^n(d^*_n(t))|^2 = \mathbb{E}\int_{d_n(t)}^t \Vert  G(s,y^n(s))\Vert _{\mathcal{T}_2(K,H)}^2\,ds \leq \mathbb{E}\int_{d_n(t)}^t [K_0 + K_1 |y^n(s)|^2]\,ds.
\]
Therefore, Fubini's theorem and \eqref{E1_S2} yield
\begin{eqnarray}\label{y-u*}
 \mathbb{E}\int_0^T |y^n(t)-u^n(d^*_n(t))|^2 dt &\leq &
 C\mathbb{E}\int_0^T\,ds [1+|y^n(s)|^2] \int_s^{d^*_n(s)} dt \nonumber \\
&\leq & C \frac{T}{n}.
\end{eqnarray}

We next prove that for any $ \eps\; \in [0,1)$, we have
\begin{equation}\label{u-u_H}
\mathbb{E}\int_{0}^{T}|u^{n}(d_{n}^{*}(t)^{-})-u^{n}(t)|^{2}dt\leq \frac{CT}{n}.
\end{equation}
This estimate together with \eqref{y-u*} concludes the proof of \eqref{y-u} when $ \eps\; =0$.
The evolution equation \eqref{equality_u_2} shows that
\begin{equation*}
|u^{n}(d_{n}^{*}(t)^{-})-u^{n}(t)|^{2}
=2\int_{t}^{d_{n}^{*}(t)}\lb u^{n}(s)-u^{n}(t), du^{n}(s)\rb=
\sum_{i=1}^{3}T_{i}(t),
\end{equation*}
where
\begin{align*}
T_{1}(t)&=-2(1-\eps)\int_{t}^{d_{n}^{*}(t)}\lb  Au^{n}(s), u^{n}(s)-u^{n}(t)\rb\,ds, \\
T_{2}(t)&=-2\int_{t}^{d_{n}^{*}(t)}\lb  B(u^{n}(s),u^{n}(s)) ,u^{n}(s)-u^{n}(t)\rb\,ds,\\
T_{3}(t)&=-2\int_{t}^{d^*_{n}(t)}\lb  R(s,u^{n}(s)), u^{n}(s)-u^{n}(t)\rb\,ds.
\end{align*}

Using Lemma \ref{l1_S2} and the Young inequality, we deduce that
\begin{align*}
\left | \mathbb{E} \int_{0}^{T}T_{1}(t)dt\right | &=\left | (1-\eps)\mathbb{E} \int_{0}^{T}dt \int_{t}^{d_{n}^{*}(t)}
\left [-2\Vert u^{n}(s)\Vert ^{2}+2\Vert u^{n}(s)\Vert \Vert u^{n}(t)\Vert  \right ]ds \right |\\
&\leq \left | (1-\eps)\mathbb{E} \int_{0}^{T}dt \int_{t}^{d_{n}^{*}(t)}
\left [-2\Vert u^{n}(s)\Vert ^{2}+2\Vert u^{n}(s)\Vert ^{2}+ \frac{1}{2}\Vert u^{n}(t)\Vert ^{2} \right ]ds \right |\\
&\leq  \frac{1-\eps}{2}
\mathbb{E} \int_{0}^{T}dt \Vert u^{n}(t)\Vert ^{2}\int_{t}^{d_{n}^{*}(t)}ds
\; \leq\;  \frac{CT}{n}.
\end{align*}

Furthermore, using the upper estimate \eqref{B0},
the Cauchy-Schwarz inequality, the Fubini Theorem,
 Lemmas  \ref{l1_S2} and \ref{l4_S2},   and \eqref{inter}, we deduce
\begin{align*}
\Big| \mathbb{E} \int_{0}^{T}T_{2}(t)dt\Big | &\leq 2\int_{0}^{T}dt \int_{t}^{d_{n}^{*}(t)}
| u^{n}(s)|_{\mathrm{X}}^{2}\Vert u^{n}(t)\Vert\,ds\\
&\leq 2\Big (\mathbb{E}\int_{0}^{T}\Vert u^{n}(t)\Vert ^{2}dt\Big)^{1/2}
\Big(\mathbb{E}\int_{0}^{T}dt \Big[\int_{t}^{d^*_{n}(t)}\vert u^{n}(s)\vert^{2}_{\mathrm{X}}ds\Big]^{2}\Big)^{1/2}\\
&\leq C\Big(\mathbb{E}\int_{0}^{T}dt \frac{T}{n}\int_{t}^{d_{n}^{*}(t)}\vert u^{n}(s)\vert^{4}_{\mathrm{X}}ds\Big)^{1/2}\\
&\leq C \Big(\frac{T}{n} \; \mathbb{E}\int_{0}^{T}\,ds \vert u^{n}(s)\vert^{4}_{\mathrm{X}}\int_{d_n(s)}^s dt   \Big)^{1/2} \leq \frac{CT}{n}.
\end{align*}
Using Assumption {\bf (R1)}, the Cauchy-Schwarz inequality,  the Fubini Theorem and \eqref{E1_S2}, we have:
\begin{align*}
\Big| \mathbb{E} \int_{0}^{T} & T_{3}(t)dt\Big | \leq  C\mathbb{E}\int_{0}^{T}dt \int_{t}^{d_{n}^{*}(t)}
\left[ 1+|u^{n}(s)|^{2}\right ]ds\\
&+C\Big(\mathbb{E}\int_{0}^{T}|u^{n}(t)|^{2}dt\Big)^{1/2}
\Big (\mathbb{E}\int_{0}^{T}dt\Big[ \int_{t}^{d_{n}^{*}(t)}(1+|u^{n}(s)|)ds\Big]^{2}\Big )^{1/2}\\
&\leq \mathbb{E}\int_{0}^{T}\,ds[1+ |u^{n}(s)|^{2}] \int_{d_n(s)}^s dt
+C \Big(\mathbb{E}\int_{0}^{T}dt \frac{T}{n}\int_{t}^{d_{n}^{*}(t)}(1+|u^{n}(s)|)^{2}ds\Big)^{1/2}\\
&\leq \frac{CT}{n}+C \Big( \frac{T}{n}\mathbb{E}\int_{0}^{T}\,ds (1+|u^{n}(s)|)^{2} \int_{d_n(s)}^s dt \Big)^{1/2}
\; \leq \;  \frac{CT}{n}.
\end{align*}
This concludes the proof of \eqref{u-u_H} and hence of \eqref{y-u} when $ \eps\; =0$.

{\bf Case 2:} Suppose that $ \eps\; \in (0,1)$. Then for $t\in (0,T]$ we have
\[ y^n(t) - u^n(t) =- \int_t^{d^*_n(t)} \!\! F_\eps(s,u^n(s))\,ds -  \eps\; \int_{d_n(t)}^t \!A y^n(s)\,ds + \int_{d_n(t)}^t \!G(s,y^n(s)) dW(s).\]
Therefore, the It\^o Lemma  implies  that
\[  \mathbb{E}\int_0^T  |y^n(t) - u^n(t)|^2 dt  =  \sum_{i=1}^5 \bar{T}_i, \]
where
\begin{align*}
\bar{T}_1&=   2 (1-\eps) \mathbb{E}\int_0^T \!\!dt \int_t^{d^*_n(t)} \!\! \lb  A u^n(s) , y^n(s)- u^n(s)\rb \,ds,\\
\bar{T}_2&= 2 \mathbb{E}\int_0^T \!\!dt \int_t^{d^*_n(t)} \!\! \lb B(u^n(s),u^n(s)), y^n(s)-u^n(s)\rb\,ds,\\
\bar{T}_3&= 2  \mathbb{E}\int_0^T \!\!dt \int_t^{d^*_n(t)} \!\! \lb R(s, u^n(s)), y^n(s)-u^n(s)\rb\,ds, \\
 \bar{T}_4 &=  -2 \eps\;  \mathbb{E} \int_0^T \!\!dt \int_{d_n(t)}^t \! \lb A y^n(s), y^n(s)-u^n(s)\rb \,ds,\\
 \bar{T}_5& =  \mathbb{E}\int_0^T \!\!dt \int_{d_n(t)}^t \Vert G(s,y^n(s))\Vert _{\mathcal{T}_2(K,H)}^2\,ds.
\end{align*}

Let us note that since $A=A^\ast $ is non negative, we have, for all $y,u \in D(A)$:
\begin{eqnarray}\label{ineq-new01}
\langle  A u , y- u\rb &=&\lb  A (u-y) , y- u\rb+\lb  A y , y- u\rb \leq  \lb  A y , y- u\rb\, , \\
\nonumber
2 \langle A u , y- u\rb &\leq & \langle A u , y- u\rb  +\lb  A y , y- u\rb=\lb  A (y+u) , y- u\rb\\
\label{ineq-new02}
&=& \lb  A y , y\rb-\lb  A u , u\rb \leq \lb  A y , y\rb= \Vert y \Vert^2.
\end{eqnarray}

Therefore, the Fubini Theorem, the  Cauchy Schwarz inequality,
and the  estimate
\eqref{E3y_S2} yield
\[ \bar{T}_1 \leq (1-\eps) \mathbb{E}\int_0^T  ds \,\, \Vert y^n(s)\Vert ^2 \int_{d_n(s)}^s \,dt  \leq C\frac{1-\eps}{n}. \]
Similarly, the Cauchy Schwarz inequality and the upper estimates  \eqref{E1_S2} and
\eqref{E3y_S2} yield
\[ \bar{T}_4 \leq 2\;\eps \; {\mathbb E} \int_0^T\,ds \,  \big[ \,2\,\Vert y^n(s)^2 + \Vert u^n(s)\Vert ^2\big] \int_{d_n(s)}^s dt \leq  C\frac{\eps}{n}.\]
The Fubini Theorem and the upper estimates \eqref{B1} with $\eta=1$, \eqref{E1_S2},
 \eqref{E3y_S2}, \eqref{E4u_S2}, \eqref{B0}  and \eqref{inter} yield
\[ \bar{T}_2\leq C\mathbb{E}\int_0^T \,ds
\big[\,\Vert y^n(s) - u^n(s)\Vert ^2 + C_1 \vert u^n(s)\vert_\mathrm{X}^4\big]  \int_{d_n(s)}^s dt \leq C\frac{T}{n}.\]
Using Assumption {\bf (R1)},   the Fubini Theorem and \eqref{E1_S2}, we deduce that
\[ \bar{T}_3\leq C\mathbb{E} \int_0^T ds \,\, [1+|u^n(s)|^2]  \int_{d_n(s)}^s \,dt \leq C \frac{T}{n}.\]
Finally, Assumption {\bf (G1}, \eqref{E3y_S2} and the Fubini Theorem yield
\[ \bar{T}_5\leq \mathbb{E}\int_0^T ds \,  [K_0+K_1|y^n(s)|^2 +  \;\eps\; K_2 \Vert y^n(s)\Vert ^2] \int_{d_n(s)}^{s}  \, dt \leq C\frac{T}{n}.\]
This concludes the proof of \eqref{y-u} when $ \eps\; \in (0,1)$.
\end{proof}
\bigskip
 In section \ref{sec-convergence} we will prove that the scheme $u^n$ converges
 to the true  solution $u$ of  equation \eqref{weak}   in probability in $H$, with rate of convergence 
  almost equal to $1/2$. However, in order to prove this
 convergence, we have to obtain estimates similar to that proved in section \ref{sec3}, 
 shifting regularity; this will be done in
 section \ref{sect-periodic}.  Note that the results
 proved so far hold for a general hydrodynamical models with more general boundary conditions, while the results obtained
 in sections \ref{sect-periodic} and \ref{sec-convergence} require more smoothness on the coefficients and
 the initial condition, as well as periodic boundary conditions
 on some specific domain $D$.

 \section{Stochastic Navier-Stokes Equations with periodic boundary conditions}\label{sect-periodic}

In this section we will discuss the stochastic NSEs with periodic (in space) boundary conditions, or as it is usually called,
 on a $2$-D torus.
 The difference between the periodic boundary conditions and the Dirichlet ones investigated in the previous section
  is that in addition to the properties
  satisfied for both cases, the periodic case has also one additional \eqref{eqn-A-B}, see below.
 This assumption allows one not only to solve the problem with the initial datum $u_0$
  belonging to $V$ but also to get a priori bounds similar to that
 of \eqref{thm-estimate} for the solution
  and to that of Lemmas \ref{l1_S2} and \ref{l4_S2} for the scheme, shifting regularity by one level.

All what we have discussed throughout the paper until now applies to the case when the Dirichlet boundary conditions are replaced by the periodic boundary conditions.
In the latter case it is customary to study our problem  in the $2$-dimensional torus $\mathbb{T}^2$ (of fixed dimensions $L\times L$),
instead of a regular bounded domain $D$. All the mathematical background can be found in the small book \cite{Temam_1983} by Temam.  In particular, the space $\rH$ is equal to
\[\rH=\{ u\in \mathbb{L}_0^2 : \divv (u)=0 \mbox{ and } \gamma_{\mathbf{n}}(u)_{\vert \Gamma_{j+2}}=-\gamma_{\mathbf{n}}(u)_{\vert \Gamma_{j}}, \; j=1,2 \} , \]
where $\gamma_n$ is defined in \eqref{Temam_1.18}  and
$\mathbb{L}_0^2=L_0^2(\mathbb{T}^2,\mathbb{R}^2)$ is the Hilbert space consisting of those $u\in L^2(\mathbb{T}^2,\mathbb{R}^2)$
which satisfy $\int_{\mathbb{T}^2} u(x)\, dx=0$ and
$\Gamma_j$, $j=1,\cdots,4$ are the four (not disjoint) parts of the boundary of $\partial(\mathbb{T}^2)$ defined by
\[
\Gamma_j=\{ x=(x_1,x_2) \in [0,L]^2: x_j=0\},\;\Gamma_{j+2}=\{ x=(x_1,x_2) \in [0,L]^2: x_j=L\},\;\; j=1,2.
\]
Similarly, the space $\rV$ is equal to
\[\rV=\{ u\in \mathbb{L}_0^2\cap H^{1,2}(\mathbb{T}^2,\mathbb{R}^2): \divv u=0 \mbox{ and } u_{\vert \partial(\mathbb{T}^2)}=0 \}. \]

The Stokes operator $\rA$ can be defined in a natural way and it satisfies all the properties known in the bounded domain case.
In particular $A$ is positive and the following identity involving the  Stokes operator $A$ and  the nonlinear term $B$ holds:
\begin{eqnarray}\label{eqn-A-B}
\left<\rA u,B(u,u)\right>_\rH=0, \;\; u\in D(\rA);
\end{eqnarray}
see \cite[Lemma 3.1]{Temam_1983} for a proof.

 We will also need to strengthen the assumptions on the initial condition $u_0$ and on the coefficients $G$ and $R$ to obtain
 a uniform control of the $V$-norm of the solution. This is done in the following subsection.

\subsection{A priori estimates for the initial data in $V$ for the Stochastic NSEs on a torus}\label{subsec-torus}

Given $u\in V$, recall that we define $\mbox{\rm curl} \, u
= \partial_{x_1}u_{2}-\partial_{x_2}u_{1}$.  The following results are classical; see e.g. \cite{Yu}, or \cite{BesMil}
 where they are used in  a stochastic framework.\\
\begin{eqnarray}
 |\Delta u|^{2}&=& |\nabla  \mbox{\rm curl } \, u|^{2}, \quad \mbox{\rm for } u\in Dom(A),
 \label{deltacurl}\\
 |\nabla u|_{L^2} &\leq &C   |\mbox{\rm curl } u|_{L^2} , \quad \mbox{\rm for } u\in V, \label{gradientcurl}\\
\lb \mbox{\rm curl }  B(u,u),v \rb
&=&  \lb B(u ,\mbox{\rm curl } u ), v\rb,\quad \mbox{\rm for } u,v\in Dom(A).
\label{curlB}
\end{eqnarray}
To ease notations, we denote by $| { \rm curl}\, u |$ the $L^2$-norm of the one-dimensional function ${\rm curl}\, u $.

Suppose that  the coefficients $G$ and $R$ satisfy the following assumptions:

\noindent {\bf Assumption (G2):} $G :  [0,T]\times D(A)\to \mathcal{T}_2(K, V)$  (resp. $G : [0,T]\times  V\to \mathcal{T}_2(K, H)$ if $ \eps\; =0$),
and there exist positive constants
$K_i, i=0,1,2$ and $L_i, i=1,2$ such  that   for every $t\in [0,T]$, and $u,v\in D(A)$ (resp. $u,v\in V$):
 \begin{eqnarray} | \mbox{\rm curl } G(t,u)|^2_{\mathcal{T}_2(K,V)} &\leq& K_0+ K_1 |
{\rm curl}\, u|^2 +  \;\eps\; K_2 |A u|^2,  \label{growthG2} \\
|\mbox{\rm curl } G(t,u) - \mbox{\rm curl } G(t,v)|^2_{\mathcal{T}_2(K,V)} &\leq& L_1
|{\rm curl}\,  (u-v) |^2 +  \;\eps\; L_2 |Au-Av|^2.\label{LipG2}
\end{eqnarray}

\noindent {\bf Assumption (R2)}: Let us assume that $R$ is a measurable mapping
$R: [0,T]\times V\longrightarrow V$ such that for some positive constants  $R_{0}$ and $R_{1}$
\begin{equation}\label{R2}
|\mbox{\rm curl  R} (t,u)|_{L^2}\leq R_{0},\;  |\mbox{\rm curl } [R(t,u)-R (t,v)]|_{L^2} \leq R_{1}|\mbox{\rm curl }(u-v)|_{L^2},
\; u, v\in V.
\end{equation}

Let us put $\xi(t)=\mbox{\rm curl } u(t)$ where $u$ is the solution to \eqref{SNS}. Then $\xi$ solves the following equation on $[0,T]$
with initial condition  $\xi(0)=\mbox{\rm curl u}_0$:
\[ d\xi(t) + \big[A \xi(t) +  \mbox{\rm curl  B}(u(t),u(t)) + \mbox{\rm curl } R(t,u(t))\big] dt = \mbox{\rm curl G}(u(t)) \,dW(t).\]
An easy modification of the arguments in the proof of Proposition 2.2 in \cite{BesMil}  proves the following.

\begin{theorem}\label{BesMil}
Let us  assume that $u_{0}$ is a $V$-valued, ${\mathcal F}_0$-measurable
random variable with  $\mathbb{E}\Vert u_0\Vert ^{2p} <\infty$ for some real number $p\geq 2$. Assume that
that assumptions {\bf (G1)} and {\bf (G2)} are satisfied with $K_2<\frac{2}{2p-1}$ and $L_2<2$, and that the assumptions
{\bf (R1)} and {\bf (R2)} hold true.
Then the process $u$ solution of  \eqref{SNS} is such that
$u\in C([0,T]; V)\bigcap L^{2}(0,T; D(A))$ {\rm  a.s.}
Moreover, there exists a positive constant $C$ such that
\begin{equation} \mathbb{E}\Big( \sup_{t\in [0,T]}\Vert u(t)\Vert ^{2p}+\int_{0}^{T}|Au(s)|^{2}\big( 1+\Vert u(s)\Vert ^{2(p-1)} \big)\,ds \Big)\leq C(1+\mathbb{E} \Vert u_{0}\Vert ^{2p}).
\label{SNSinV}
\end{equation}
\end{theorem}

In the rest of this subsection we  suppose that the coefficients $G$ (resp. $R$) satisfy both Assumptions {\bf (G1)}, {\bf (G2)} (resp. {\bf (R1)} and {\bf (R2)}.
This will enable us to upper estimate the $V$-norm of the difference $y^n-u^n$, and hence to strengthen the inequality \eqref{y-u}.

For every $t\in [0,T)$ let $\xi^n(t)={\rm curl } \, u^n(t)$
and $\eta^n(t)= {\rm curl }\,  y^n(t)$.
Equations \eqref{deterministic-inductive}  and \eqref{stochastic-inductive}  imply that
$\xi^n(t)=\eta^n(t)= {\rm curl }\, u_0$ for $t\in [t_{-1},t_0)$  and for $i=0, \cdots, n-1$ and $t\in [t_i, t_{i+1})$, we have
\begin{align}\label{det-induc-curl}
&d \xi^n(t)  + \big[ (1-\eps) A\xi^n(t) + {\rm curl }\, B(u^n(t),u^n(t)) + {\rm curl }\, R(t,u^n(t))\big] dt =0, \nonumber\\
&\xi^n(t_i^+) = \eta^n(t_i^-),
\end{align}
and
\begin{align}\label{sto-induc-curl}
&d \eta^n(t)  +  \;\eps\; A\eta^n(t) dt =  {\rm curl }\, G(t,y^n(t)) \,dW(t), \nonumber\\
&\eta^n(t_i^+)= \xi^n(t_{i+1}^-).
\end{align}
The processes $(\xi^n(t), t_i\leq t<t_{i+1})$ (resp. $(\eta^n(t), t_i\leq t<t_{i+1})$) are well-defined processes
 which are ${\mathcal F}_{t_i}$-measurable
(resp. $({\mathcal F}_t)$-adapted). Furthermore, these equations  can be reformulated as follows for $t\in [0,T]$:
\begin{align}\label{equality_xi_2}
\xi^n(t)&={\rm curl }\, u_{0}- \!\int_{0}^{t} \!\!{\rm curl }\, F_{\eps}(s,u^n(s))ds
+\int_{0}^{d_{n}(t)}\!\!\left[ - \eps\; A \eta^n(s)ds+{\rm curl }\, G(s,y^n(s))dW(s)\right],\\
\label{equality_eta_2}
\eta^n(t)&={\rm curl }\,u_{0}- \!\int_{0}^{d_{n}^{*}(t)} \!\!\!{\rm curl }\, F_\eps(s,u^n(s))ds+\int_{0}^{t}\!\!
\left[ - \eps\; A \eta^n(s)ds+{\rm curl }\, G(s,y^n(s))dW(s)\right].
\end{align}

We at first prove the following analog of Lemmas \ref{l1_S2} and \ref{l4_S2}.

\begin{lemma}\label{momentsxieta} Assume that $p$ is a  integer such that $p\geq 2$.
Assume that   $G$ satisfies Assumptions {\bf (G1)} and  {\bf (G2)} with $K_2<\frac{2}{2p-1}$  and $L_2<2$,
and  $R$ satisfies Assumptions {\bf (R1)} and {\bf(R2)}.
Let $u_0$ be an ${\mathcal F}_0$-measurable, $V$-valued random variable  such that $\mathbb{E}\Vert u_0\Vert ^{2p}<\infty$.
Then for $ \eps\; \in [0,1)$, there exists a positive constant $C$ such that for every integer $n\geq 1$
\begin{equation} \label{xin}
 \sup_{t\in [0,T]}\mathbb{E} \big(  \Vert u^n(t)\Vert ^{2p} + \Vert y^n(t)\Vert ^{2p} \big)
+ \mathbb{E}\int_0^T \big( 1+ \Vert u^n(t)\Vert ^{2(p-1)} \big) |A u^n(t)|^{2} dt \leq C.
\end{equation}
Furthermore, if $ \eps\; \in (0,1)$, there exists a positive constant $C$ such that for every integer $n\geq 1$
\begin{equation}\label{etan}
 \mathbb{E}\int_0^T \big( 1+ \Vert y^n(t)\Vert ^{2(p-1)}) |A y^n(t)|^2 dt \leq C.
\end{equation}
\end{lemma}
\begin{proof}

We briefly sketch the proof, which is similar to that of Lemmas \ref{l1_S2} and \ref{l4_S2}. Let us fix $n\in \mathbb{N}$.
First note that, using Lemmas  \ref{l1_S2} and \ref{l4_S2}, \eqref{deltacurl} and \eqref{gradientcurl}, it is easy to see that
the upper estimates \eqref{xin} and \eqref{etan}
can be deduced from similar upper estimates where $\Vert  u^n\Vert $, $\Vert y^n\Vert $, $| A u^n |$ and $| A y^n |$ are replaced by
$|\xi^n|$, $|\eta^n|$, $\Vert  \xi^n\Vert $ and $\Vert \eta^n\Vert $ respectively.
We prove by induction that for $l\in \{ -1, 0, \cdots, n-1\}$,
\begin{align}
\Big[ \mathbb{E}\Big( \sup_{t\in [t_l,t_{l+1})} &|\xi^n(t)|^2\Big) \vee \Big( \sup_{t\in [t_l,t_{l+1})} \mathbb{E}|\eta^n(t)|^2\Big) \Big]
\leq \mathbb{E}|{\rm curl }\, u_0 |^2 e^{\frac{(l+1)aT}{n}} +
\frac{\alpha T}{n} \sum_{j=1}^{l+1} e^{\frac{jaT}{n}}, \label{majo-induc-2}\\
\Big[ \mathbb{E}\Big( \sup_{t\in [t_l,t_{l+1})} &|\xi^n(t)|^{2p}\Big) \vee \Big( \sup_{t\in [t_l,t_{l+1})} \mathbb{E}|\eta^n(t)|^{2p}\Big) \Big]
\leq \mathbb{E}|{\rm curl}\, u_{0}|^{2p}{\rm e}^{\frac{(l+1)bT}{n}}+\frac{\beta T}{n}\sum_{j=1}^{l+1}{\rm e}^{\frac{jbT}{n}},
\label{majo-induc-4}
\end{align}
where    $a:=2(R_{0}+R_{1})+K_{1}$, $\alpha:=2R_{0}+K_{0}$,
$b:=2p(R_{0}+R_{1})+p(2p-1)(K_{0}+K_{1})$ and $\beta:=2pR_{0}+p(2p-1)K_{1}$.
Indeed, these inequalities hold for $l=-1$. Suppose that they hold for $l\leq i-1$, $i<n-1$; we prove them for $l=i$.

We at first prove \eqref{majo-induc-2} for $l=i$.
The identities \eqref{det-induc-curl},  \eqref{curlB} and Assumption {\bf (R2)} imply that for $t\in [t_i, t_{i+1})$, we have
\begin{align*}
|\xi_n(t)|^2 &+ 2(1-\eps) \int_{t_i}^t \Vert \xi^n(s)\Vert ^2\,ds = |\eta^n(t_i^-)|^2 -2 \int_{t_i}^t \lb {\rm curl} \, R(s,u^n(s)), \xi^n(s)\rb\,ds\\
&\leq |\eta^n(t_i^-)|^2 + 2 \frac{R_0 T}{n} + 2 (R_0+R_1)\int_{t_i}^t  |\xi^n(s)|^2\,ds,
\end{align*}

while the identities \eqref{sto-induc-curl} and \eqref{deltacurl}, the It\^o Lemma  and Assumption {\bf (G2)} imply
\begin{align*}
\mathbb{E}|\eta^n(t)|^2 &+ 2 \eps\; \mathbb{E}\int_{t_i}^t \Vert  \eta^n(s)\Vert ^2\,ds = \mathbb{E}|\xi^n(t_{i+1}^-)|^2 + \mathbb{E}\int_{t_i}^t \Vert  {\rm curl} \, G(s,y^n(s))\Vert _{\mathcal{T}_2(K,H)}^2 \,ds\\
&\leq \mathbb{E}|\xi^n(t_{i+1}^-)|^2 + \frac{K_0 T}{n} + K_1  \mathbb{E}\int_{t_i}^t |\eta^n(s)|^2 \,ds +  \eps\; K_2 \mathbb{E}\int_{t_i}^t \Vert   \eta^n(s)\Vert ^2]\,ds.
\end{align*}
Then arguments similar to that used in the proof of Lemma \ref{l1_S2} yield \eqref{majo-induc-2} and then
for every $ \eps\; \in [0,1)$ and $C$ independent of $n$
\[ \sup_{t\in [0,T)}\mathbb{E} \big(  \Vert u^n(t)\Vert ^2 + \Vert y^n(t)\Vert ^2\big)  + \sup_{n\in\mathbb{N}} \mathbb{E}\int_0^T  |A u^n(t)|^2 dt \leq C.\]
Furthermore, for $ \eps\; \in (0,1)$ we have
\[ \mathbb{E}\int_0^T  |A y^n(t)|^2 dt \leq C.\]
We then prove \eqref{majo-induc-4}  for $l=i$.
  Using once more   \eqref{curlB} and Assumption {\bf (R2)}, we deduce that
 for $t\in [t_i, t_{i+1})$, we have
 \begin{align*}
|\xi^n(t)|^{2p} &+ 2p (1-\eps)\! \int_{t_i}^t \! \Vert \xi^n(s)\Vert ^2 |\xi^n(s)|^{2(p-1)}\,ds\\
& =
|\eta^n(t_i^-)|^{2p} -2p\!  \int_{t_i}^t \! \lb {\rm curl} \, R(s,u^n(s)), \xi^n(s)\rb |\xi^n(s)|^{2(p-1)} \,ds\\
&\leq |\eta^n(t_i^-)|^{2p}  + 2p \frac{R_0T}{n} + 2p (R_0+R_1)  \int_{t_i}^t |\xi^n(s)|^{2p}\,ds.
\end{align*}
Similarly, \eqref{sto-induc-curl}, the It\^o Lemma  and Assumption {\bf (G2)} imply that for $t\in [t_i, t_{i+1})$, we have
\begin{align*}
\mathbb{E}&|\eta^n(t)|^{2p} + 2p  \eps\; \mathbb{E}\int_{t_i}^t \Vert  \eta^n(s)\Vert ^2 |\eta^n(s)|^{2(p-1)}  \,ds \\
&=
 \mathbb{E}|\xi^n(t_{i+1}^-)|^{2p} + 2p(p-1) \mathbb{E}\int_{t_i}^t \Vert  {\rm curl} \, G(s,y^n(s))\Vert _{\mathcal{T}_2(K,H)}^2  |\eta^n(s)|^{2(p-1)}\,ds\\
&\leq \mathbb{E}|\xi^n(t_{i+1}^-)|^{2p} + 2p(p-1) \frac{K_0 T}{n} + 2p(p-1)(K_0+K_1) \mathbb{E}\int_{t_i}^t \!\! |\eta^n(s)|^{2p}\,ds\\
&\qquad + 2p(p-1) \eps\; K_2 \mathbb{E}\int_{t_i}^t \!\! \Vert \eta^n(s)\Vert ^2 |\eta^n(s)|^{2(p-1)}\,ds.
\end{align*}
Then arguments similar to those used in Lemmas \ref{l1_S2} and \ref{l4_S2} yield \eqref{majo-induc-4},
and then \eqref{xin} and \eqref{etan}. This concludes the proof of the proposition since $(2p-1)K_2<2$ when $ \eps\; >0$.
\end{proof}
We finally find an estimate from above on the  $V$ norm of the difference of $u^n$ and $y^n$. This estimate will be used to obtain
the speed of convergence of the scheme.
\begin{proposition}\label{prop-yn-un_V}
Assume  that $u_0$ is ${\mathcal F}_0$-measurable with  $\mathbb{E}\Vert u_0\Vert ^4 <\infty$; let $T>0$ and $ \eps\; \in [0,1)$.
Assume that  Assumptions {\bf (G1)} and {\bf (G2)} are satisfied  with $K_2 < 2/3$ and $L_2<2$,
and  that
Assumptions {\bf (R1)} and {\bf (R2)} hold.
 Then there exists a positive constant $C:=C(T)$,  such that for every integer  $n\geq 1$
\begin{equation} \label{yn-un_V}
\mathbb{E}\int_0^T \Vert  y^n(t) - u^n(t)\Vert ^2 dt \leq \frac{C}{n}.
\end{equation}
\end{proposition}
\begin{proof}
Using \eqref{y-u} and \eqref{gradientcurl}, we see that the proof of \eqref{yn-un_V} reduces to check that
\begin{equation} \label{xin-etan}
\mathbb{E}\int_0^T |\eta^n(t) - \xi^n(t)|^2 dt \leq \frac{C}{n}.
\end{equation}
We only prove this inequality when $ \eps\; \in (0,1)$. The proof in the case  $ \eps\; =0$ can be done by  adapting  the arguments in the proof of Proposition
\ref{prop_y-u}. So let us assume that  $ \eps\; \in (0,1)$. Then   the It\^o Lemma,
\eqref{xin} and \eqref{etan} imply that for any $t\in (0,T)$:
\begin{align*}
&\mathbb{E}|\eta^n(t)-\xi^n(t)|^2= -2 \; \mathbb{E}  \int_t^{d^*_n(t)} \!\! \Big[ (1-\eps) \lb A\xi^n(s),\eta^n(s)-\xi^n(s)\rb \\
&\qquad \qquad +  \lb B(u^n(s),\xi^n(s)), \eta^n(s) - \xi^n(s)\rb + \lb {\rm curl }\, R(s, u^n(s)) , \eta^n(s)-\xi^n(s)\rb \Big]\,ds \\
& \quad -2 \eps\; \mathbb{E} \int_{d_n(t)}^t \lb A \eta^n(s), \eta^n(s)-\xi^n(s)\rb
+\mathbb{E} \int_{d_n(t)}^t \Vert  {\rm curl }\, G(s,y^n(s))\Vert _{\mathcal{T}_2(K,H)}^2\,ds
\end{align*}
Integrating on $[0,T]$, using \eqref{B1} with $\eta =1$,
Assumptions {\bf (R2)} and {\bf (G2)},  the Fubini Theorem and Lemma \ref{momentsxieta},
we obtain
\begin{align*}
\mathbb{E}\! \int_0^T & \!\! \!|\eta^n(t)-\xi^n(t)|^2 dt \leq 2 \mathbb{E}\! \int_0^T\!\! \!\,ds \Big[ (1-\eps) |A u^n(s)| |A (y^n(s)-u^n(s))|  +  |A(u^n(s)- y^n(s)|^2 \\
& + C_1\vert u^n(s)\vert_\mathrm{X}^2 \vert \xi^n(s)\vert_\mathrm{X}^2 \Big] \int_{d_n(s)}^s \!\! dt
+ 2 \eps\; \mathbb{E}\int_0^T \!\! \,ds (|A u^n(s)|^2 + |Ay^n(s)|^2)  \int_s^{d^*_n(s)} \!\! dt \\
&+ \mathbb{E}\int_0^T \!\,ds [K_0+K_1 |\eta^n(s)|^2 + K_2  \eps\; |A y^n(s)|^2] \int_s^{d^*_n(s)} \!\! dt\\
\leq & \frac{CT}{n} \mathbb{E}\int_0^T \!\!\big( |A u^n(s)|^2 + |A y^n(s)|^2 + \Vert  u^n(s)\Vert ^4 + |u^n(s)|^4 + |\xi^n(s)|^2 |Au^n(s)|^2\big)\,ds \\
\leq & \frac{CT}{n},
\end{align*}
where the last inequality can be deduced from Lemmas \ref{l4_S2} and  \ref{momentsxieta}. This concludes the proof.
\end{proof}

\subsection{A priori estimates for the process $z^n$}

For technical reasons, let  us consider the process $(z^n(t), t\in [0,T])$  defined by
\begin{equation}\label{defzn}
z^n(t)=u_0 - \int_0^t F_\eps(s,u^n(s))\,ds -  \eps\; \int_0^{d_n(t)} \! A y^n(s)\,ds
+ \int_{0}^{t} G(s,y^n(s))dW(s).
\end{equation}
Note that for any $ \eps\in [ 0,1)$ the process $z^n(t)$ coincides with $u^n(t^+)$ and $y^n(t^-)$ on the time grid, i.e.,
\begin{equation}\label{uyz_n}
z^n(t_k)=y^n(t_k^-)=u^n(t_k^+) \quad \mbox{\rm for} \; k=0,1, \cdots, n.
\end{equation}
The following lemma gives upper estimates of the differences $z^n-u^n$ and $z^n-y^n$ in various topologies.

\begin{lemma} \label{lem_zn-un}
Let $ \eps\; \in [0,1)$ and let $u_0$ be ${\mathcal F}_0$-measurable.

(i) Assume that $u_0$ is $H$-valued with ${\mathbb E} |u_0|^{2p}<\infty$ for some integer $p\geq 2$. Suppose  that   Assumption  {\bf (G1)}  holds  with
$K_2<\frac{2}{2p-1}$ and $L_2<2$  and  that Assumption {\bf (R1)} is satisfied. Then there exists a positive constant $C:=C(T, \eps)$
such that for every integer $n\geq 1$,
\begin{equation}\label{z-u_H}
 \sup_{t\in [0,T]}  \mathbb{E}|z^n(t)- u^n(t)|^{2p} \leq \frac{C}{n^p}\, .
 \end{equation}

 (ii) Assume that $u_0$ is $V$-valued with ${\mathbb E} \Vert u_0\Vert ^{4}<\infty$,  let Assumptions {\bf (G1)} and {\bf (G2)} hold with
 $K_2<\frac23$  and $L_2<2$ and let Assumptions {\bf (R1)} and {\bf (R2)} be satisfied.
Then there exists a positive constant $C:=C(T, \eps)$ such that for every integer $n\geq 1$,
 \begin{equation} \label{z-u_L2V}
 \mathbb{E}\int_0^T \big( \Vert z^n(t)-u^n(t)\Vert ^2 + \Vert z^n(t)-y^n(t)\Vert ^2\big) dt \leq \frac{C}{n}\, .
 \end{equation}
 \end{lemma}

\begin{proof}
For $t\in [0,T]$, we have
\[ z^n(t)-u^n(t)=\int_{d_n(t)}^t  G(s,y^n(s)) dW(s).\]
Let the assumptions of (i) be satisfied. Since  $\mathbb{E}|u_0|^{2p}<\infty$
 and $K_2<\frac2{2p-1}$, the Burkholder-Davies-Gundy and H\"older inequalities together
 with Assumption {\bf (G1)}  and  Lemma
\ref{l4_S2}
imply that for any $t\in [0,T]$,
\begin{align*}
\mathbb{E}|z^n(t)&-u^n(t)|^{2p} \leq C_p \mathbb{E}\Big| \int_{d_n(t)}^t \Vert G(s,y^n(s)\Vert ^2_{\mathcal{T}_2(K,H)}\,ds \Big|^p\\
&\leq C_p \Big(\frac{T}{n}\Big)^{p-1} \mathbb{E} \int_{d_n(t)}^t \Big| K_0 + K_1 |y^n(s)|^2 +  \eps\;  K_2 \Vert y^n(s)\Vert ^2 \Big|^p\,ds\\
&\leq \frac{C_p(T)}{n^{p-1}} \Big[ K_0^p  + K_1^p \sup_{t\in [0,T]}  \mathbb{E}|y^n(t)|^{2p}
+ \eps^p \; K_2^p \sup_{t\in [0,T]}  \mathbb{E}\Vert y^n(t)\Vert ^{2p}\Big] \frac{T}{n}
\leq \frac{C_p(T)}{n^p}.
\end{align*}
Let the assumptions of (ii) hold true and  let $\zeta^n(t)={\rm curl } \, z^n(t)$; then
\begin{align*} \zeta^n(t)\; =\; & {\rm curl } \, u_0 -\int_0^t \big[ (1-\eps) A\zeta^n(s) +
 {\rm curl } \,B(u^n(s),u^n(s)) + {\rm curl } \,R(s,u^n(s))\big]\,ds\\
&- \eps\;  \int_0^{d_n(t)} A\eta^n(s)\,ds + \int_0^t {\rm curl } \,  G(s,y^n(s)) dW(s).
\end{align*}
Therefore, $\zeta^n(t)-\xi^n(t)=\int_{d_n(t)}^t {\rm curl } \, G(s,y^n(s)) dW(s)$; the It\^o Lemma, the Fubini Theorem and
Assumption {\bf (G2)}  imply that
\begin{align*}
\int_0^T \mathbb{E}|\zeta^n(t)-\xi^n(t)|^2\,ds &= \int_0^T \mathbb{E}\int_{d_n(t)}^t \big| {\rm curl } \, G(s,y^n(s))\big|_{\mathcal{T}_2(k,H)}^2\,ds dt\\
&\leq \int_0^T \big[K_0+K_1 |\eta^n(s)|^2 +  \eps\;  K_2 |A y^n(s)|^2\big] \Big( \int_s^{d^*_n(s)} dt\Big)\,ds\\
&\leq \frac{T}{n} \Big[ K_0 T +  \mathbb{E}\int_0^T \big( K_1 |\eta^n(t)|^2 +  \eps\;  K_2 |A y^n(t)|^2\big) dt\Big].
\end{align*}
Hence, Lemma \ref{momentsxieta} and \eqref{z-u_H} applied with $p=1$ yield
\[ \int_0^T \mathbb{E}\Vert z^n(t)-u^n(t)\Vert ^2 dt \leq \frac{C(T)}{n}. \]
This upper estimate and Proposition \ref{prop-yn-un_V} conclude the proof of \eqref{z-u_L2V}.
\end{proof}
\bigskip

\section{Speed of Convergence for the scheme in the case of the Stochastic NSEs
on a $2$D torus } \label{sec-convergence}
In this entire section we consider the Stochastic NSEs  on a $2$D torus.
Thus, we suppose that  $u_0$ is ${\mathcal F}_0$-measurable such that
$\mathbb{E}\Vert u_0\Vert ^8<\infty$ and fix some $T>0$.
 We also assume that  Assumptions {\bf (G1)} and {\bf (G2)} are satisfied with $ K_2 < \frac{2}{7}$ and $ L_2<2$,
 and that the Assumptions
{\bf (R1)} and {\bf (R2)} hold true.
Our aim  is to  prove  the convergence of the scheme in the $H$ norm uniformly on the time grid and in $L^2(0,T;V)$.

The non linearity of the Navier-Stokes Equations
 requires   to impose a localization in order to obtain an $L^2(\Omega, {\mathbb P})$ convergence.
However, the probability of this localization set converges to 1, which yields the order 
of convergence in probability as defined
by Printems in \cite{Prin}; the order of convergence in probability of $u^n$ to $u$ in $H$
is $\gamma >0$ if
\[ \lim_{C\to \infty}  P\Big( \sup_{k=0, \cdots, n} |u^n(t_k)-u(t_k)| \geq \frac{C}{n^\gamma}\Big) =0.
\]
First, we prove that when properly localized, the schemes $u^n$ and $y^n$ converge to $u$ in $L^2(\Omega)$ for various topologies.
Thus, for every $M>0$,  $t\in [0,T]$ and any integer $n\geq 1$, let
\begin{equation}\label{Omega_M}
\Omega_{M}^n(t):=\left\{\omega\in\Omega\; : \; \int_{0}^{t}\big( \vert u(s,\omega)\vert _{\mathrm{X}}^{4}
+ \vert u^n(s,\omega)\vert _{\mathrm{X}}^{4}\big) \,ds\leq M\right\}.
\end{equation}
This definition shows that  $\Omega_M^n(t)\subset \Omega_M^n(s)$ for $s\leq t$ and that
$\Omega_M^n(t)\in {\mathcal F}_t$ for any
$t\in [0,T]$.  Furthermore, Lemma \ref{l1_S2} shows that  $\sup_{n\geq 1} P(\Omega_M^n(T)^c)\to 0$ as $M\to \infty$.
Let $\tau^n_M=\inf\{ t\geq 0\; : \; \int_{0}^{t}\big( \vert u(s,\omega)\vert _{\mathrm{X}}^{4}
+ \vert u^n(s,\omega)\vert _{\mathrm{X}}^{4}\big)\,ds\geq M\}\wedge T$; we clearly have $\tau_M^n=T$ on the set $\Omega_M^n(T)$.

The following proposition proves that, localized on the set $\Omega_M^n(T)$, the strong speed of convergence of  $z^n$  to $u$ (resp.
of $u^n$ and $y^n$ to $u$) in
$L^\infty(0,T;H)$ (resp. in $ L^2(0,T;V)$) is $1/2$.
\begin{Prop}\label{lzn-u}
Let  $  \eps\;  \in [0,1)$, $u_0$ be ${\mathcal F}_0$-measurable such that $\mathbb{E}\Vert u_0\Vert ^8<\infty$.
 Suppose  that the Assumptions {\bf (G1)}, {\bf (G2)}, {\bf (R1)}
and {\bf (R2)} hold with $K_2<\frac{2}{7}$,  $L_2<2$  with
 $ \eps\;  L_2$  strictly smaller than $2(1-\eps)$.
Then  there exist positive  constants  $C(T)$  and $\tilde{C}_2$ such that for every for every
$M>0$  and   $n\in\mathbb{N}$,  we have:
\begin{equation} \label{majo_zn-u}
\mathbb{E}\Big( \sup_{t\in [0, \tau^n_M]} \! |z^n(t)-u(t)|^2 + \int_0^{\tau^n_M} \!\!
\big[ \Vert u^n(t)-u(t)\Vert ^2 + \Vert y^n(t)-u(t)\Vert ^2\big] dt\Big)\leq \frac{K(T,M)}{n},
\end{equation}
where $z^n$ is defined by
\eqref{defzn}, $u$ is the solution to \eqref{SNS} and
\[ K(M,T)=C(T) \exp\big( C(T)  e^{\tilde{C}_2 M}\big). \]
\end{Prop}
\begin{proof} First of all let us observe that in view of Proposition \ref{prop-yn-un_V} we only have to prove inequality \eqref{majo_zn-u} for the first two terms on the left hand side.  For this aim
let us fix $M>0$ and a natural number $n\geq 1$. To ease notation we put $\tau:=\tau_M^n$. Then  we have
\begin{align}
z^n(t\wedge \tau)-u(t\wedge \tau)=&-\int_0^{t\wedge \tau} [F_\eps(s,u^n(s)) - F(s,u(s))]\,ds -  \eps\;  \int_0^{d_n(t\wedge \tau)}
\!\! A y^n(s)\,ds  \nonumber \\
& + \int_0^{t\wedge \tau} [ G(s,y^n(s)) - G(s,u(s)) ] dW(s), \mbox{ for any $t\in [0,T]$}. \label{zn-u}
\end{align}
The It\^o Lemma  yields  that
\[|z^n(t\wedge \tau)-u(t\wedge \tau )|^2 = \sum_{i=1}^4 T_i(t) + I(t),\;\; \mbox{ for }t\in [0,T],\]
 where, also for $t\in [0,T]$,

\begin{align}
I(t)\; =\; & 2\int_0^{t\wedge \tau} \lb [G(s,y^n(s)) - G(s,u(s))] \,dW(s)\, , \, z^n(s)-u(s)\rb,   \label{It}\\
T_1(t)\; =\; & -2\int_0^{t\wedge \tau} \lb F_{\eps} (s,u^n(s)) - F_\eps(s,u(s)) \, , \, z^n(s)-u(s)\rb\,ds, \nonumber \\
T_2(t)\; =\; & -2  \eps\;  \int_0^{d_n(t\wedge \tau)} \lb Ay^n(s) - A u(s) \, , \, z^n(s)-  u(s)\rb \,\,ds, \nonumber \\
T_3(t)\; =\; & -2 \eps\;  \int_{d_n(t\wedge \tau)}^{t\wedge \tau}  \lb A u(s) \, , \, z^n(s)-u(s)\rb\,ds, \nonumber \\
T_4(t)\; =\; & \int_0^{t\wedge \tau} |G(s, y^n(s)) - G(s,u(s)) |_{\mathcal{T}_2(K,H)}^2\,ds.  \label{decompos_z_n-u}
\end{align}
We at first upper estimate the above term $T_1(t)$ as follows: $T_1(t)\leq \sum_{i=1}^5 T_{1,i}(t)$, where
\begin{align}
T_{1,1}(t) \,=\,  & -2(1-\eps) \int_0^{t\wedge \tau} \lb A u^n(s) - Au(s) \, , \, u^n(s)-u(s)\rb\,ds, \nonumber \\
T_{1,2}(t) \, = \, & -2(1-\eps) \int_0^{t\wedge \tau} \lb A u^n(s) - Au(s) \, , \, z^n(s)-u^n(s)\rb\,ds, \nonumber \\
T_{1,3}(t)\, =\, & -2\int_0^{t\wedge \tau} \lb B(u^n(s),u^n(s)) - B(u(s),u(s)) \, , \, u^n(s)-u(s)\rb\,ds, \nonumber \\
T_{1,4}(t)\, =\, & -2\! \int_0^{t\wedge \tau} \!\! \!  \Big\langle \big[ B(u^n(s)- u(s) ,u^n(s))
+ B(u(s),u^n(s)- u(s))  ,z^n(s)-u^n(s)\Big\rangle\,ds, \nonumber \\
T_{1,5}(t)\, =\, & -2(1-\eps) \int_0^{t\wedge \tau} \lb R(s, u^n(s)) - R(s,u(s)) \, , \, z^n(s)- u(s)
\rb\,ds.
\label{decompT1}
\end{align}
The definition of the $V$-norm implies that
\begin{equation} \label{T11}
T_{1,1}(t)=-2(1-\eps) \int_0^{t\wedge \tau} \Vert u^n(s)-u(s)\Vert ^2\,ds ,
\end{equation}
while the Cauchy-Schwarz and Young inequalities yield that for any $\eta >0$:
\begin{align}
T_{1,2}(t)&\leq  2(1-\eps) \int_0^{t\wedge \tau} \Vert u^n(s)-u(s)\Vert \; \Vert  z^n(s)-u^n(s)
\Vert \,\,ds \nonumber \\
&\leq \eta (1- \eps\; ) \int_0^{t\wedge \tau} \Vert u^n(s)-u(s)\Vert ^2\,ds + \frac{1-\eps}{\eta} \int_0^{t\wedge \tau} \Vert z^n(s)-u^n(s)\Vert ^2\,ds.
\label{T12}
\end{align}
Using \eqref{Blip} we deduce that for any $\eta>0$ we have:
\begin{align}
T_{1,3}(t)\leq & 2\eta \int_0^{t\wedge \tau} \Vert u^n(s)-u(s)\Vert ^2\,ds + 2 C_\eta \int_0^{t\wedge \tau} |u^n(s)-u(s)|^2 \vert u(s)\vert_\mathrm{X}^4\,ds
\nonumber \\
\leq & 2\eta \int_0^{t\wedge \tau}\Vert u^n(s)-u(s)\Vert ^2\,ds + 4 C_\eta \int_0^{t\wedge \tau} |z^n(s)-u(s)|^2 \vert u(s)\vert_\mathrm{X}^4\,ds\nonumber \\
& +4  C_\eta \int_0^{t\wedge \tau} |z^n(s)-u^n(s)|^2 \vert u(s)\vert_\mathrm{X}^4\,ds. \label{T13}
\end{align}
Furthermore, condition \eqref{B1} on $B$, property \eqref{inter} of the space $\mathrm{X}$ and  the Cauchy-Schwarz and Young inequalities yield for any $\eta >0$:

\begin{align}
T_{1,4}&(t)\leq  2 \int_0^{t\wedge \tau} \!\! \big[ 2\eta \|z^n(s)-u^n(s) \|^2 + C_\eta |u^n(s)-u(s)|_{ X}^2 \big( |u^n(s)|_{ X}^2
+ | u(s)|_{X}^2\big)\big] ds \nonumber \\
\leq & 8\eta \int_0^{t\wedge \tau}\|u^n(s)-u(s)\|^2 ds + 8\eta \int_0^{t\wedge \tau}\|z^n(s)-u(s)\|^2 ds \nonumber \\
& \;  + 2C C_\eta  \Big( \int_0^{t\wedge \tau}\|u^n(s)-u(s)\|^2 ds \Big)^{1/2} \Big( \int_0^{t\wedge \tau}|u^n(s)-z^n(s)|^4 ds \Big)^{1/4}
\nonumber \\
&\qquad \times
\Big( \int_0^{t\wedge \tau}  \big[ |u^n(s)|_{ X}^8  + |u(s)|_{ X}^8 \big] ds \Big)^{1/4}
\nonumber \\
&\; + 2 C C_\eta \Big( \int_0^{t\wedge \tau}\!\!\!  \|u^n(s)-u(s)\|^2 ds \Big)^{1/2}  \Big( \int_0^{t\wedge \tau} \!\! \! |z^n(s)-u(s)|^2
\big[ |u(s)|_{ X}^4 + |u^n(s)|_{ X}^4\big] ds \Big)^{1/2}\nonumber \\
 \leq &10\eta \int_0^{t\wedge \tau}\|u^n(s)-u(s)\|^2 ds +  8\eta \int_0^{t\wedge \tau}\|z^n(s)-u(s)\|^2 ds \nonumber \\
& + C \frac{C_\eta}{\eta}  \Big( \int_0^{t\wedge \tau}\!\!\! |u^n(s)-z^n(s)|^4 ds\Big)^{1/2} \Big( \int_0^{t\wedge \tau}\!\!\!
\big[ |u(s)|_{ X}^8 + |u^n(s)|_{X}^8\big] ds \Big)^{1/2} \nonumber \\
&+ C \frac{C_\eta}{\eta} \int_0^{t\wedge \tau}\!\!\! |z^n(s)-u(s)|^2 \big[ |u(s)|_{ X}^4 + |u^n(s)|_{ X}^4\big] ds .
\label{T14}
\end{align}

Finally, Assumption {\bf (R1)} and the triangular and  Cauchy Schwarz inequalities imply
\begin{align}
T_{1,5}(t)\leq & 2R_1 \int_0^{t\wedge \tau}|u^n(s)-u(s)|\, |z^n(s)-u(s)|\,ds\nonumber \\
\leq & 3 R_1 \int_0^{t\wedge \tau}|z^n(s)-u(s)|^2\,ds + R_1 \int_0^{t\wedge \tau}|z^n(s)-u^n(s)|^2\,ds. \label{T15}
\end{align}

First note that $T_2=-2\epsilon  \int_0^{t\wedge \tau} \langle \nabla (y^n(s)-u(s))\, , \,  \nabla (z^n(s)-u(s)) \rangle ds$.
Replacing $u$ by $u^n$, and using the Cauchy-Schwarz and Young inequalities we deduce
\begin{align}
T_2(t) \leq & - 2 \eps\! \int_0^{d_n(t\wedge \tau)} \!\! \Vert u^n(s)-u(s)\Vert ^2\,ds
+ 2 \eps\!  \int_0^{d_n(t\wedge \tau)}\!\!  \Vert y^n(s)-u^n(s)\Vert  \Vert z^n(s)-u^n(s)\Vert \,ds
\nonumber \\
& +  2 \eps\;  \int_0^{d_n(t\wedge \tau)} \Vert u^n(s)-u(s)\Vert  \big[ \Vert y^n(s)-u^n(s)\Vert   +\Vert z^n(s)-u^n(s)\Vert \big]\,ds
\nonumber \\
\leq & \; 2 \eps\;  \int_0^{d_n(t\wedge \tau)} \Vert y^n(s)-u^n(s)\Vert  \Vert z^n(s)-u^n(s)\Vert \,ds.
\label{T2}
\end{align}

The Cauchy Schwarz and Young inequality imply that for every $\eta >0$, we have
\begin{align}
T_3(t)\leq & 2 \eps\;  \int_{d_n(t\wedge \tau)}^{t\wedge \tau} \Vert u(s)\Vert  \big( \Vert z^n(s)-u^n(s)\Vert  + \Vert  u^n(s)-u(s)\Vert \big)\,ds \nonumber \\
\leq &  \eps\;  \eta \int_{d_n(t\wedge \tau)}^{t\wedge \tau} \Vert u^n(s)-u(s)\Vert ^2\,ds \nonumber \\
& + \frac{\eps}{\eta} \Big( \int_{d_n(t\wedge \tau)}^{t\wedge \tau} \Vert u(s)\Vert ^2\,ds \Big)^{1/2} \Big( \int_{d_n(t\wedge \tau)}^{t\wedge \tau}
\Vert z^n(s)-u^n(s)\Vert ^2\,ds \Big)^{1/2}. \label{T3}
\end{align}
Finally, since $ \eps\;  L_2 < 2(1-\eps) $ we can choose $\alpha >0$ such that $ \eps\;  L_2 < 2(1-\eps) -(2+L_2)\alpha$;
thus the Assumption {\bf (G1)} yields
\begin{align}
T_4(t)& \leq  L_1 \int_0^{t\wedge \tau} \!\! \big[ |y^n(s)-u(s)|^2 + \eps\;  L_2 \Vert  y^n(s)-u(s)\Vert ^2\big]\,ds \nonumber \\
&  \leq  2 L_1 \int_0^{t\wedge \tau} \!\! |z^n(s)-u(s)|^2\,ds +  \eps\;  L_2(1+\alpha) \int_0^{t\wedge \tau} \!\!
\|u^n(s)-u(s)\|^2\,ds
\nonumber\\
& \; +2 L_1 \int_0^{t\wedge \tau}  \!\! |y^n(s)-z^n(s)|^2\,ds
 +  \eps\;  L_2 \Big(1+\frac{1}{\alpha}\Big)
\int_0^{t\wedge \tau} \!\! \Vert y^n(s)-u^n(s)\Vert ^2\,ds.  \label{T4}
\end{align}
Fix $\eta >0$ such that   $13 \eta <\alpha$.   Put
 \[X(t):=\sup_{s\in [0,t\wedge \tau ]} |z^n(s)-u(s)|^2\, , \;\;
Y(t):=\int_0^{t\wedge\tau} \Vert u^n(s)-u(s)\Vert ^2\,ds, \mbox{ for } t\in [0,T]. \]
Then  inequalities \eqref{zn-u}--\eqref{T4} imply that
\[
X(t)+ \alpha Y(t)\leq \int_0^{t\wedge \tau}\!\!  \varphi(s) X(s)\,ds + \sup_{s\in [0,t\wedge \tau]} |I(t)| + Z(t),
\mbox{ for  $t\in [0,T]$,}
\]
where  $I(t)$ is defined by  \eqref{It}, while the processes  $\varphi$ and $Z$ are defined as follows:
\begin{align*}
\varphi(s)=& 4C_\eta \vert u(s)\vert_\mathrm{X}^4 + C \frac{C_\eta}{\eta} \big( \vert u(s)\vert_\mathrm{X}^4 +\vert u^n(s)\vert_\mathrm{X}^4 \big)
+3R_1+2L_1, \\
Z(t)=&  \int_0^{t\wedge \tau} \Big[  \Big( \frac{1-\eps}{\eta} + 2 \eps +4 L_{1}\; \Big)
\Vert z^n(s)-u^n(s)\Vert ^2
+  |z^n(s)-u^n(s)|^2 \big( R_1+ 4 C_\eta \vert  u(s)\vert_\mathrm{X}^4)\\
&   + \Big[ \eps\;  L_2 \Big(  1+ \frac{1}{\alpha}\Big) + 2 \eps +4L_{1}\; \Big ]
\Vert y^n(s)-u^n(s)\Vert ^2 \Big]\,ds \\
& + \frac{C_\eta}{\eta} C \Big( \int_0^{t\wedge \tau} |u^n(s)-z^n(s)|^4\,ds\Big)^{1/2}
\Big(\int_0^{t\wedge \tau} [\vert u^n(s)\vert_\mathrm{X}^8 + \vert u(s)\vert_\mathrm{X}^8]\,ds
\Big)^{1/2}  \\
&+ \frac{\eps}{\eta} \Big( \int_{d_n(t)\wedge \tau}^{t\wedge \tau} \Vert u(s)\Vert ^2\,ds\Big)^{1/2}
\Big( \int_{d_n(t)\wedge \tau}^{t\wedge \tau} \Vert z^n(s)-u^n(s)\Vert ^2\,ds\Big)^{1/2}.
\end{align*}
The definition of the stopping time $\tau$ implies the existence of constants $C_1$ and $C_2$
 larger than 1 and independent of $n$ such that:
\[ \int_0^\tau  \varphi(s)\,ds \leq (3R_1+2L_1)T + M\big(4C_\eta + C \frac{C_\eta}{\eta}\big)=C_1T+C_2M:= C(\varphi).\]
Since $V\subset \mathrm{X}$,  using Propositions \ref{prop_y-u}, \ref{prop-yn-un_V} and  
Lemma \ref{lem_zn-un},
we deduce the existence of a constant
$C(T)$ depending on $T$ (and not on $n$) such that

\begin{align}
\mathbb{E}[Z(T)]\leq & \Big[ \frac{1-\eps}{\eta} + R_1+4 \eps\;  + 8L_1 +  \eps\;  L_2 \big(1+\frac1\alpha\big)\Big] \frac{C}{n} \nonumber \\
& +  4 C_\eta \Big( \mathbb{E}\int_0^T |z^n(s)-u^n(s)|^4\,ds\Big)^{1/2} \Big( {\mathbb E} \int_0^T \Vert u(s)\Vert ^8\,ds\Big)^{1/2} \nonumber\\
& + C \frac{C_\eta}{\eta} \Big( \mathbb{E}\int_0^{T\wedge \tau} |u^n(s)-z^n(s)|^4\,ds\Big)^{1/2}
\Big( \mathbb{E}\int_0^{T\wedge \tau} \big( \Vert u^n(s)\Vert ^8 + \Vert u(s)\Vert ^8 \big)\,ds\Big)^{1/2} \nonumber\\
&+ \frac\eps\eta \Big( \frac{T}{n} \sup_{s\in [0,T]} \mathbb{E}\Vert u(s)\Vert ^2 \Big)^{1/2} \Big(\mathbb{E}\int_0^T \Vert z^n(s)-u^n(s)\Vert ^2ds\Big)^{1/2}
\leq \frac{C(T)}{n}. \label{majoEZ}
\end{align}
Furthermore, Assumption {\bf(G1)}, the Burkholder-Davies-Gundy,   the Cauchy-Schwarz and the Young inequalities   imply that for any $\beta >0$:

\begin{align} \label{majoI(t)}
\mathbb{E}\Big( & \sup_{s\leq t \wedge \tau}|I(s)|\Big) \leq 12
\mathbb{E}\Big( \int_0^{t\wedge \tau} \Vert G(s,y^n(s))- G(s,u(s))\Vert _{\mathcal{T}_2(K,H)}^2
|z^n(s)-u(s)|^2\,ds \Big)^{1/2}\nonumber \\
&\leq \beta \mathbb{E}\Big(\sup_{s\leq t\wedge \tau} |z^n(s)-u(s)|\Big)
+ \frac{36}{\beta} \mathbb{E}\int_0^{t\wedge \tau} \!\!\! \big[ L_1|y^n(s)-u(s)|^2 +
 \eps\;  L_2 \Vert y^n(s)-u(s)\Vert ^2\big]\,ds  \nonumber \\
&\leq \beta \mathbb{E}X(t) + \frac{36 L_1 (1+\delta_1)}{\beta} \int_0^t \mathbb{E} X(s)\,ds
+ \frac{36  \eps\;  L_2 (1+\delta_2)}{\beta}\; \mathbb{E} Y(t)
+ \frac{\tilde{C}}{n},
\end{align}
where $\tilde{C}=C \Big( \frac{36 L_1}{\beta (1+\delta_1)}
+ \frac{36  \eps\;  L_2}{\beta (1+\delta_2)}\Big) $.
Choose $\beta>0$ such that $2\beta \big(1+C(\varphi) e^{C(\varphi)}\big) =1$.
 Then suppose that $ \eps\;  L_2$ is small enough to ensure that
  \[72 (1+\delta_2)  \eps\;  \frac{1}{\beta} L_2\big(1+ C(\varphi) e^{C(\varphi)}\big)  \leq \alpha.\] Using the argument similar to that used
  in the  proof of Lemma 3.9 in \cite{DuanMil} we deduce that
\[ X(t)+\alpha Y(t)\leq \big[ Z+\sup_{s\leq t\wedge \tau} | I(s) | \big] \big(1+C(\varphi) e^{C(\varphi)}  \big).\]
Then taking expectation, using \eqref{majoI(t)},  the Gronwall lemma and \eqref{majoEZ}, we deduce
\begin{align*}
\mathbb{E} X(T) + {\alpha} \mathbb{E}Y(T) &\leq  2\big( 1 + C(\varphi)e^{C(\varphi)}\big)\Big[  \mathbb{E} (Z) + \frac{\tilde C}{n} \Big]
\exp\Big[ {144} L_1 (1+\delta_1) \big(1+C(\varphi) e^{C(\varphi)} \big)^2\Big]
\\
&\leq \frac{C(T)+\tilde{C}}{n}  \exp\big[  C(T) e^{3 C_2 M}\big]
\end{align*}
for some positive constant $C(T)$ which does not depend on $n$. This completes the proof of \eqref{majo_zn-u}
\end{proof}

The following theorem proves that, properly localized, the sequences $u^n$ and $y^n$
converge strongly to $u$ in $L^2(\Omega, {\mathbb P})$
for various topologies, and that the "localized" speed of convergence of these processes  is $1/2$.

\begin{theorem}\label{th-speed}
Let $  \eps\;  \in [0,1)$; suppose that Assumptions {\bf (R1)}, {\bf (R2)}, {\bf (G1)}  and {\bf (G2)} hold with $\ L_2$ small enough and
$K_2<\frac27$, and let $u_0$ be ${\mathcal F}_0$-measurable with $\mathbb{E}\Vert  u_0\Vert ^8<\infty$. Then
the processes $u^n$ and $y^n$ defined in section \ref{desscheme} converge to the solution $u$ to the stochastic
Navier-Stokes Equations \eqref{SNS} on a $2$-D torus.
 More precisely, given any $M>0$ there exist   positive constants $C(T)$ and $\tilde{C}_2$,  which do not depend on $n$ and $M$,
 such that for    any integer  $n= 1, \cdots$  we have
\begin{align}\label{speed_sup}
& \mathbb{E}\Big[ 1_{\Omega_M^n( T  )} \sup_{k=0, \cdots, n-1}
\Big( \sup_{s\in [t_k,t_{k+1})} \big( |u^n(s^+) - u(s)|^2 + |y^n(s^+)-u(s)|^2 \big)  \Big)  \Big]  \leq \frac{K(M,T)}{n}
\\
&  \mathbb{E}\Big[ 1_{\Omega_M^n(t)}  \int_0^t \!\!\big[ \Vert  u^n(s) - u(s)\Vert ^2  +  \Vert y^n(s)-u(s)\Vert ^2\big]ds \Big)  \Big]
\leq \frac{K(M,T)}{n}, \label{speed-L2}
\end{align}
where  $K(M,T):=C(T) \exp\big[ C(T)  e^{3M}\big]$,
 and
 \[ \Omega_M^n(t)= \Big\{ \omega : \int_0^t \big( \vert  u(s)(\omega)\vert_\mathrm{X}^4 + \vert u^n(s)(\omega)\vert ^4_\mathrm{X}\big)\,ds
\leq M \Big\} \; \mbox{\rm  for  $t\in [0,T]$.}\]
\end{theorem}

\begin{proof} First note that on $\Omega_M^n(t)$  we have  $\tau^n_M\geq t$.
Hence using Propositions \ref{lzn-u},
we deduce that \eqref{speed-L2} holds true.
Furthermore, the Cauchy Schwarz inequality and \eqref{etan} prove that
\begin{align*}
\mathbb{E}\Big( \sup_{k=1, \cdots, n} |z^n(t_k)-y^n(t_k^-)|^2 \Big) &=
\mathbb{E}\Big( \sup_{k=1, \cdots, n} \eps^2 \Big| \int_{t_k}^{t_{k+1}} A y^n(s)\,ds \Big|^2 \Big)  \\
& \leq \eps^2 \frac{T}{n} \mathbb{E}\Big( \sup_{k=1, \cdots, n}
  \int_{t_k}^{t_{k+1}} \big|A y^n(s)\big|^2\,ds  \Big)
\leq \frac{C(T) \eps^2}{n}.
\end{align*}
Therefore, since $ z^n(t_k) =u^n(t_k^+)=y^n(t_k^-)$, Proposition \ref{lzn-u} yields
\begin{equation} \label{speed-supgrid}
\mathbb{E} \Big[ 1_{\Omega_M^n(T )} \sup_{k=0, \cdots, n} \big( |u^n(t_k^+) - u(t_k)|^2 + |y^n(t_k^-)-u(t_k)|^2 \big)\Big] \leq
\frac{K(M,T)}{n}.
\end{equation}
Finally, using Assumption {\bf (G1)}, Young's inequality and Lemma \ref{momentsxieta}, we obtain that for any $k=0, \cdots, n-1$:
\begin{align*}
\mathbb{E}\Big(& \sup_{t\in (t_k,t_{k+1})} |y^n(t)-y^n(t_k^+)|^2 \Big) \leq
 \mathbb{E}\Big[  \int_{t_k}^{t_{k+1}}\!\! \big( \frac{\eps}{2}  \Vert y^n(t_k^-)\Vert ^2
+  \big[ K_0+K_1 |y^n(s)|^2 +  \eps\;  K_2 \Vert y^n(s\Vert ^2\big]\big) \,ds \\
&\qquad + \Big(\int_{t_k}^{t_{k+1}} |y^n(s)-y^n(t_k^+)|^2 \big[ K_0 + K_1 |y^n(s)|^2 + K_2 \;  \eps\;  \Vert y^n(s)\Vert ^2\big]\,ds \Big)^{1/2}\Big]
 \\
&\leq \frac{1}{2} \mathbb{E}\Big( \sup_{t\in (t_k,t_{k+1})} |y^n(t)-y^n(t_k^+)|^2 \Big) + C \frac{T}{n} \sup_{s\in [0,T]} \mathbb{E}\Vert y^n(s)\Vert ^2,
\end{align*}
so that
\begin{equation} \label{yn-yn}
\mathbb{E}\Big( \sup_{t\in (t_k,t_{k+1})} |y^n(t)-y^n(t_k^+)|^2 \Big) \leq \frac{C(T)}{n}.
\end{equation}
Using the It\^o Lemma, the  inequalities \eqref{B0} and \eqref{SNSinV}, the Schwarz and Young inequality, the  Assumptions {\bf (G1)}
and the Burkholder Davies Gundy inequality, we deduce that
\begin{align*}
\mathbb{E}\Big( & \sup_{t\in (t_k, t_{k+1}) }|u(t)-u(t_k^+)|^2\Big) \leq
\mathbb{E}\Big( \sup_{t\in (t_k, t_{k+1}) } \int_{t_k}^t \Big[ - 2 \Vert u(s)\Vert ^2 + 2 \Vert  u(s) \Vert   \Vert  u(t_k^+)\Vert  \\
& \qquad + \vert  u(s)\vert_\mathrm{X}^2 \Vert  u(s)-u(t_k^+)\Vert ^2
+(R_0+R_1 |u(s)|)\,  |u(s)-u(t_k^+)| \\
& \qquad + \big( K_0+K_1 |u(s)|^2 +  \eps\;  K_2 \Vert  u(s) \Vert ^2\big) \Big]\,ds \Big) \\
&\quad  + \mathbb{E} \Big( \int_{t_k}^{t_{k+1}} |u(s)-u(t_k^+)|^2
\big[ K_0+K_1 |u(s)|^2 +  \eps\;  K_2 \Vert u(s)\Vert ^2 \big]\,ds \Big)^{1/2}\\
& \leq \frac{1}{2} \mathbb{E}\Big(  \sup_{t\in (t_k, t_{k+1}) }|u(t)-u(t_k^+)|^2\Big)  +\frac{C}{n} \Big( 1+ \sup_{t\in [0,T]} \mathbb{E} \Vert u(s)\Vert ^4 \Big),
\end{align*}
and hence
\begin{equation} \label{u-u}
\mathbb{E}\Big(  \sup_{t\in (t_k, t_{k+1}) }|u(t)-u(t_k^+)|^2\Big) \leq \frac{C(T)}{n} .
\end{equation}
A similar simpler argument using the inequalities \eqref{B0} and \eqref{xin} yields
\begin{equation}  \label{un-un}
 \mathbb{E}\Big(  \sup_{t\in (t_k, t_{k+1}) }|u^n(t)-u^n(t_k^+)|^2\Big) \leq \frac{C}{n} \Big( 1+ \sup_{t\in [0,T]} [ \mathbb{E} \Vert u^n(s)\Vert ^4
+ \Vert  y^n(t)\Vert ^4] \Big)\leq \frac{C(T)}{n}.
\end{equation}
The inequalities \eqref{yn-yn}--\eqref{un-un} and \eqref{speed-supgrid} conclude the proof of \eqref{speed_sup}.
\end{proof}

\begin{corollary}\label{cv-proba-positive}
Let $  \eps\;  \in [0,1)$; assume that the assumptions of Theorem \ref{th-speed} are satisfied.
For any integer $n\geq 1$ let  $\tilde{e}_n(T)$  denote the error term defined by
\[ \tilde{e}_n(T)=\sup_{k=1,\cdots,n } |u^n(t_k) - u(t_k)| + \Big( \int_0^T \!\! \Vert  u^n(s) - u(s)\Vert ^2 \,ds \Big)^{1/2}
+ \Big(\int_0^T \!\! \Vert y^n(s)-u(s)\Vert ^2\,ds \Big)^{1/2}.\]
Then $\tilde{e}_n(T)$ converges to 0 in probability with the speed almost $1/2$.  To be precise,  for any sequence $\big(z(n)\big)_{n=1}^\infty $  converging  to $\infty$,
\begin{equation} \label{speed-proba-0}
\lim_{n \to \infty} \mathbb{P}\Big( \tilde{e}_n(T)\geq \frac{z(n)}{\sqrt{n}}\Big) = 0.
\end{equation}
Therefore, the scheme $u^n$ converges to $u$ in probability in $H$ with rate 1/2.
\end{corollary}
\begin{proof}
Let $z(n)\to \infty$ and let $M(n):=\ln(\ln( \ln (z(n))))$; then  $M(n)\to \infty$. Thus,
using \eqref{thm-estimate} and \eqref{xin} for $p=2$ and the Markov inequality, we
deduce that $\mathbb{P} (\Omega^c_{M(n)}(T)) \to 0$ as $n\to \infty$ .
Finally, note that $C(T) \big(\ln ( \ln z(n))\big)^3 - \ln\big(z(n) \big)  \to - \infty$ as $n\to \infty$
for any positive constant $C(T)$.
Therefore, using the inequalities \eqref{speed_sup}
and \eqref{speed-L2},
the explicit forms of ${K}(T,M(n))$, the choice of  $M(n)$
and Markov's inequality, we deduce that
\begin{align*}
\mathbb{P}\Big( \tilde{e}_n(T)\geq \frac{z(n)}{\sqrt{n}}\Big)\leq &
\mathbb{P}\big(\Omega^c_{M(n)}(T)\big) + \frac{n}{z(n)^2} \mathbb{E}\Big(1_{\Omega_{M(n)}(T) } \tilde{e}_n(T)^2\Big)\\
\leq &  \mathbb{P}\big(\Omega^c_{M(n)}(T)\big) + C(T) \; \frac{n}{z(n)^2}\; \frac{1}{n}  \exp[ C(T) (\ln(\ln(z(n))))^3] \to 0
\end{align*}
as $n\to \infty$; this  concludes the proof.
\end{proof}

\textbf{Acknowledgments:} This paper was partially written while  H. Bessaih
was invited professor at the University of Paris 1 Panth\'eon-Sorbonne.
Parts of this paper were also written while the three authors were visiting the Bernoulli Center in Lausanne.
 They would like to thank the Institute for the financial
support, the very good working conditions and the friendly atmosphere.

Finally, the authors would like to thank the anonymous referees for their careful reading and valuable comments.

\end{document}